\documentclass[reqno]{amsart}
\usepackage{latexsym,amsmath,amssymb,amscd}
\usepackage[all]{xy}
\usepackage{enumerate}
\usepackage{hyperref}

\def\today{\ifcase \month \or
   January \or February \or March \or April \or
   May \or June \or July \or August \or
   September \or October \or November \or December \fi
   \space\number\day , \number\year}

\makeatletter
  \newcommand\@dotsep{4.5}
  \def\@tocline#1#2#3#4#5#6#7{\relax
     \ifnum #1>\c@tocdepth 
     \else
     \par \addpenalty\@secpenalty\addvspace{#2}%
     \begingroup \hyphenpenalty\@M
     \@ifempty{#4}{%
     \@tempdima\csname r@tocindent\number#1\endcsname\relax
        }{%
         \@tempdima#4\relax
           }%
      \parindent\z@ \leftskip#3\relax \advance\leftskip\@tempdima\relax
      \rightskip\@pnumwidth plus1em \parfillskip-\@pnumwidth
       #5\leavevmode\hskip-\@tempdima #6\relax
       \leaders\hbox{$\m@th
       \mkern \@dotsep mu\hbox{.}\mkern \@dotsep mu$}\hfill
       \hbox to\@pnumwidth{\@tocpagenum{#7}}\par
       \nobreak
        \endgroup
         \fi}
\makeatother

\newcommand{\tto}{\rightrightarrows}
\newcommand{\opn}{\operatorname}

\newcommand{\supp}{\opn{supp}}

\newcommand\pullback{^{\downarrow \downarrow}}

\renewcommand{\1}{\mathbf 1}

\newcommand{\End}{{\rm End}}

\newcommand{\Gr}{{\rm Gr}}
\newcommand{\Ind}{{\rm Ind}}
\newcommand{\Ker}{{\rm Ker}\,}
\newcommand{\spann}{{\rm span}}
\newcommand{\Tr}{{\rm Tr}}
\newcommand{\de}{{\rm d}}

\newcommand{\Ad}{{\rm Ad}}
\newcommand{\ad}{{\rm ad}}
\newcommand{\ee}{{\rm e}}
\newcommand{\id}{{\rm id}}
\newcommand{\ie}{{\rm i}}
\newcommand{\ssl}{{\mathfrak s}{\mathfrak l}}
\newcommand{\so}{{\mathfrak s}{\mathfrak o}}
\newcommand{\ssp}{{\mathfrak s}{\mathfrak p}}
\newcommand{\su}{{\mathfrak s}{\mathfrak u}}

\newcommand{\Ac}{\mathcal A}
\newcommand{\Bc}{\mathcal B}
\newcommand{\Cc}{\mathcal C}

\newcommand{\Ec}{\mathcal E}
\newcommand{\Fc}{\mathcal F}
\newcommand{\Gc}{\mathcal G}

\newcommand{\Jc}{\mathcal J}
\newcommand{\Kc}{\mathcal K}
\newcommand{\Lc}{\mathcal L}
\newcommand{\Oc}{\mathcal O}
\newcommand{\Pc}{\mathcal P}
\newcommand{\Rc}{\mathcal R}
\newcommand{\Sc}{\mathcal S}
\newcommand{\Tc}{\mathcal T}
\newcommand{\Uc}{\mathcal U}
\newcommand{\Vc}{\mathcal V}

\newcommand{\Xc}{\mathcal X}

\newcommand{\ag}{{\mathfrak a}}

\renewcommand{\gg}{{\mathfrak g}}
\newcommand{\hg}{{\mathfrak h}}
\newcommand{\kg}{{\mathfrak k}}

\newcommand{\n}{{\mathfrak n}}

\newcommand{\pg}{{\mathfrak p}}
\newcommand{\rr}{{\mathfrak r}}
\newcommand{\sg}{{\mathfrak s}}

\newcommand{\CC}{\mathbb C}

\newcommand{\RR}{\mathbb R}

\newtheorem{theorem}{Theorem}[section]
\newtheorem{proposition}[theorem]{Proposition}
\newtheorem{corollary}[theorem]{Corollary}

\newtheorem{lemma}[theorem]{Lemma}

\newtheorem{definition}[theorem]{Definition}
\newtheorem{remark}[theorem]{Remark}
\newtheorem{example}[theorem]{Example}

\numberwithin{equation}{section}

\title[$C^*$-dynamical systems]{$C^*$-dynamical systems  of solvable Lie groups}

\author[Ingrid Beltita]{Ingrid Belti\c t\u a} \address{Institute of Mathematics ``Simion Stolow'' 
of the Romanian
  Academy, PO Box 1-764, 014700 Bucharest, Romania}
\email{ingrid.beltita@gmail.com, Ingrid.Beltita@imar.ro}

\author[Daniel Beltita]{Daniel Belti\c t\u a} \address{Institute of Mathematics ``Simion Stolow'' 
of the Romanian
  Academy, PO Box 1-764, 014700 Bucharest, Romania}
\email{beltita@gmail.com, Daniel.Beltita@imar.ro}

\thanks{This work was supported by a grant of the Romanian National Authority for Scientific Research and
Innovation, CNCS--UEFISCDI, project number PN-II-RU-TE-2014-4-0370}

\date{\today}

\begin{document}

\parskip2.5pt

\begin{abstract} 
In this paper we develop a groupoid approach to 
some basic topological properties of dual spaces of solvable Lie groups 
using suitable dynamical systems related to the coadjoint action. 
One of our main results is that the coadjoint dynamical system of any exponential solvable Lie group 
is a piecewise pullback of group bundles. 
Our dynamical system approach to solvable Lie groups also allows us to construct 
some new examples of connected solvable Lie groups whose $C^*$-algebras admit faithful irreducible representations. 
\end{abstract}


\maketitle

\tableofcontents

\section{Introduction}

We initiate an approach to representation theory of solvable Lie groups that is based on the systematic use of 
groupoid $C^*$-algebras associated to the coadjoint transformation groups. 
Our main achievements in this sense are

1. a framework in which some transformation groups, including among others the 
coadjoint transformation groups of exponential solvable Lie groups, are piecewise pullbacks of group bundles 
(Theorem~\ref{stratif}, Corollary~\ref{th_exp}); 

2. a method to construct new examples of connected solvable Lie groups of type~I whose $C^*$-algebras are primitive, that is, 
admit faithful irreducible representations 
(Theorem~\ref{Pr2}, Example~\ref{Pr4}). 

To put the aforementioned results into perspective, we recall that 
if $G$ is an \emph{exponential Lie group}, 
that is, a Lie group whose exponential map $\exp_G\colon\gg\to G$ is bijective,  
then $G$ is a solvable Lie group.
It is also known that  
there exists a canonical homeomorphism from 
the space of coadjoint orbits $\gg^*/G$ onto the dual space $\widehat{G}$,  
which consists of the equivalence classes of unitary irreducible representations of $G$ (see \cite{LeLu94}). 
The dynamical system $(G,\gg^*,\Ad_G^*)$ defined by the coadjoint action $\Ad_G^*\colon G\times\gg^*\to\gg^*$ 
holds a key role in this picture by means of its orbit space $\gg^*/G$, 
and the topological properties of the orbits  encode certain properties of their corresponding representations of $G$. 
Thus, a coadjoint orbit is closed if and only if its corresponding unitary irreducible representation of~$G$ is CCR, 
that is, the corresponding image of $C^*(G)$ is equal to the $C^*$-algebra of compact operators on the representation space.  
See also \cite{BB16a} for the significance of open coadjoint orbits, whose corresponding irreducible representations 
must be regarded as the discrete series of $G$ and are square integrable cf., \cite[18.9.1]{Dix64}. 
We prove that none of these discrete series representations could be faithful on $C^*(G)$ (Proposition~\ref{even}), 
and this raises the interesting problem of describing the connected solvable Lie groups 
whose $C^*$-algebras do admit faithful irreducible representations.  
This problem is a major motivation of the present paper, as explained in Section~\ref{Sect6}, 
because the analogous primitivity problem for other algebraic structures was already solved, 
for instance for universal enveloping algebras \cite{Oo74}, for $C^*$-algebras of \'etale groupoids \cite{BCFS14}, 
and for discrete groups \cite{Mu03}, \cite{Om14}. 
On the other hand, examples of solvable Lie groups with unique open coadjoint orbits were given in \cite{Ko12}, 
and we relate them to  solvable Lie groups with primitive $C^*$-algebras.

Another theme of the present paper is to contribute to the above discussion, 
by exploring some properties of locally closed Hausdorff subsets $\Gamma\subseteq\gg^*/G$ which are related to dynamical systems, in the sense of transformation groups. 
Specifically, if $q\colon\gg^*\to\gg^*/G$ denotes the quotient map, then $\Xi:=q^{-1}(\Gamma)\subseteq\gg^*$ 
is a $G$-invariant locally compact Hausdorff space. 
One can consider the corresponding dynamical system $(G,\Xi,\Ad_G^*\vert_{G\times\Xi})$, whose orbit space is Hausdorff. 
We  develop some abstract tools that allow us to study the $C^*$-algebra of that dynamical system and its relation to topological properties of the coadjoint orbits contained in $\Gamma$. 
See Corollary~\ref{Morita4} where we establish conditions ensuring that 
a transformation group is isomorphic to a groupoid pullback of a group bundle of its isotropy groups. 
This property implies some nice properties of the $C^*$-algebra of that transformation group, 
that however could also be obtained using the theory of crossed products from \cite[Ch. 8]{Wi07}. 

Our perspective is that the natural framework to which the above themes belong is 
the theory of $C^*$-algebras of second countable, locally compact groupoids with left Haar systems, 
and therefore we establish our main results on that level of generality.

\subsection*{All groupoids are pullbacks of group bundles} 
Before proceeding with the description of the structure of the present paper, 
it is worthwhile to briefly discuss here, 
on a purely algebraic level, the method of our investigation.  
This  
may also at least in part explain the role of bundles of coadjoint isotropy groups 
in representation theory of exponential solvable Lie groups. 

Let $\Gc$ be any groupoid, viewed as a discrete topological space, 
with its domain/range maps $d,r\colon \Gc\to \Gc^{(0)}$ 
and the quotient map $q\colon\Gc^{(0)}\to\Gc^{(0)}/\Gc$ 
onto its set of orbits. 
Assume that we have fixed a map $\gamma\colon \Gc^{(0)}/\Gc\to\Gc^{(0)}$ with $q\circ\gamma=\id$, 
that is, $\gamma$ is a cross-section of~$q$. 
Denote $\Xi:=\gamma(\Gc^{(0)}/\Gc)\subseteq \Gc^{(0)}$, and fix another map $\sigma\colon \Gc^{(0)}\to \Gc$ 
with $d\circ\sigma=\id$ and $r\circ\sigma= \gamma\circ q$.
The set $\Xi$ intersects every $\Gc$-orbit at exactly one point. 

Then define the bundle of isotropy groups $\Gamma:=\bigsqcup\limits_{x\in\Xi}\Gc(x)$, 
and let $\Pi\colon \Gamma\to\Xi$ be the canonical projection with $\Pi^{-1}(x)=\Gc(x)$ for all $x\in\Xi$. 
Define $\theta:=\gamma\circ q\colon\Gc^{(0)}\to\Xi$ and the \emph{pullback of $\Pi$ by $\theta$} 
$$\theta^{\pullback}(\Pi):=\{(x,h,y)\in \Gc^{(0)}\times \Gamma\times \Gc^{(0)}\mid \theta(x)=\Pi(h)=\theta(y)\}.$$
The projections on the first and third coordinates, regarded as domain/range maps, define 
a groupoid $\theta^{\pullback}(\Pi)\tto \Gc^{(0)}$ and it is straightforward to check that 
the map 
\begin{equation}\label{introd_eq1}
\Phi\colon\Gc\to \theta^{\pullback}(\Pi),\quad 
\Phi(g):=(r(g),\sigma(r(g))g\sigma(d(g))^{-1},d(g))
\end{equation}
is a groupoid isomorphism, with its inverse  
\begin{equation}\label{introd_eq2}
\Phi^{-1}\colon \theta^{\pullback}(\Pi)\to \Gc,\quad 
\Phi^{-1}(x,h,y)=\sigma(y)^{-1}h\sigma(x)
\end{equation}
(see \cite{MRW87} and \cite{Bu03} for the special case of transitive groupoids). 
In this way every groupoid is (noncanonically) isomorphic to the pullback of the group bundle obtained as the restriction of its isotropy subgroupoid to any cross-section of its space of orbits. 
In particular, when viewed as discrete groupoids, \emph{all transformation groups are pullbacks of group bundles}. 

\subsection*{Structure of this paper}
Some of  our main results (Theorem~\ref{stratif} and its corollaries) give versions of the above italicized statement for 
some linear dynamical systems, 
including the coadjoint action of any exponential solvable Lie group. 
The key points in the proofs of these results are the conditions which ensure that the above cross-sections $\gamma$~and~$\sigma$ 
can be constructed satisfying appropriate topological conditions. 
These topological conditions should be strong enough for obtaining information on the $C^*$-algebra of the considered dynamical systems, 
in the sense that the corresponding $C^*$-algebra turns out to be piecewise Morita equivalent to $C^*$-algebras of group bundles as above. 
Along the way we found it useful to include a detailed presentation of some facts on locally compact groupoids, that may hold an independent interest. Some of these facts are already known (see for instance Theorem~\ref{Morita1}), while some others (as Proposition~\ref{strong}, or Proposition~\ref{fred0}) we were unable to locate in the literature. 
From the present perspective, the significance of these facts is that they belong to a framework in which transformation groups are obtained from 
group bundles by a procedure that preserves the Morita-equivalence class of their $C^*$-algebras. 
Therefore the bulk of this paper (Sections \ref{Sect2}--\ref{Sect4}) is devoted to establishing the topological framework in which the pullback operation works suitably. 
Then Section~\ref{Sect5} contains some of our main results  
(Theorem~\ref{stratif}, Corollary~\ref{th_exp}). 
Specific examples of solvable Lie groups, including groups whose $C^*$-algebras are primitive 
(Theorem~\ref{Pr2}, Example~\ref{Pr4}), 
are then discussed in Section~\ref{Sect6}.

\section{Preliminaries on groupoids}\label{Sect2}

\subsection{Basic notation}
\emph{Throughout the paper, unless otherwise mentioned, 
$\Gc$ is a {\bf second countable locally compact Hausdorff} 
groupoid with its space of objects $\Gc^{(0)}$, 
space of morphisms $\Gc$, and its domain/range maps $d,r\colon \Gc\to \Gc^{(0)}$.  
We usually summarize this setting by the symbol $\Gc\tto \Gc^{(0)}$.} 
For short we say that $\Gc$ is a {\bf locally compact groupoid}.  
Hence one has a category whose objects constitute the set $\Gc^{(0)}$,  
all morphisms are invertible and constitute the set $\Gc$, 
and the set of composable pairs of morphisms is defined by 
$\Gc^{(2)}:=\{(g,h)\in\Gc\times\Gc\mid d(g)=r(h)\}$; 
the sets $\Gc^{(0)}$ and $\Gc$ are endowed with second countable locally compact topologies for which the structural maps 
(domain, range, inversion of morphisms $\iota\colon\Gc\to\Gc$, and composition of morphisms $\Gc^{(2)}\to\Gc$) are continuous.  
Moreover, we assume that the canonical inclusion map $\Gc^{(0)}\hookrightarrow\Gc$, 
which maps every object to its identity morphism, is a homeomorphism onto its image, 
and the domain map $d$ is open (hence so is the range map $r=d \circ \iota$). 

For any point $x\in\Gc^{(0)}$, its isotropy group is $\Gc(x):=\{g\in\Gc\mid d(g)=r(g)=x\}$ 
and its $\Gc$-orbit is $\Gc\cdot x:=\{r(g)\mid g\in\Gc,\ d(g)=x\}=r(\Gc_x)$, where $\Gc_x:=d^{-1}(x)$.  
The set of all $\Gc$-orbits is denoted by $\Gc^{(0)}/\Gc$.  
If there exists only one orbit then $\Gc$ is called a \emph{transitive groupoid}. 
If $\Gc(x)=\{\1\}$ for all $x\in\Gc^{(0)}$, then $\Gc$ is called a \emph{principal groupoid}. 
A principal transitive groupoid is called a \emph{pair groupoid}. 

A subset $A\subseteq\Gc^{(0)}$ is said to be \emph{$\Gc$-invariant} if for every $x\in A$ 
one has $\Gc\cdot x\subseteq A$. 
If moreover $\Gc(x)=\{\1\}$ for all $x\in A$, then $A$ is called a \emph{principal} (invariant) set. 

Now assume that in addition $\Gc$ is endowed with a left Haar system, that is, 
a family $\lambda=\{\lambda^x\}_{x\in \Gc^{(0)} }$  
where $\lambda^x$ is a Radon measure supported on  $\Gc^x:=r^{-1}(x)\subseteq \Gc$ for every $x\in \Gc^{(0)} $, 
satisfying the continuity condition 
that the function $\Gc^{(0)}\ni x\mapsto\lambda(\varphi):=\int \varphi \de\lambda^x\in\CC$ is continuous 
and the invariance condition $\int \varphi(gh)\de\lambda^{d(g)}(h)=\int \varphi(h)\de\lambda^{r(g)}(h)$ 
for all $g\in\Gc$ and $\varphi\in\Cc_c(\Gc)$.

Then one can define a convolution on the space $\Cc_c(\Gc)$ 
by the formula 
$$(\varphi_1\ast \varphi_2)(g):=\int\limits_{\Gc^{r(g)}}\varphi_1(h)\varphi_2(h^{-1}g)\de\lambda^{r(g)}(h)
\text{ for }g\in\Gc\text{ and }\varphi_1,\varphi_2\in\Cc_c(\Gc).$$
This makes $\Cc_c(\Gc)$ into an associative $*$-algebra with the involution defined by  $\varphi^*(g):=\overline{\varphi(g^{-1})}$ 
for all $g\in\Gc$ and $\varphi\in\Cc_c(\Gc)$. 
There also exists a natural algebra norm on $\Cc_c(\Gc)$ defined by 
$$\Vert f\Vert_I
:=\max\Bigl\{\sup_{x\in \Gc^{(0)}}\int\vert \varphi\vert\de\lambda^x,\sup_{x\in \Gc^{(0)}}\int\vert \varphi^*\vert\de\lambda^x\Bigr\}.$$
Then $C^*(\Gc)$ is defined as the completion of $\Cc_c(\Gc)$ with respect to the norm 
$$\Vert \varphi\Vert:=\sup\limits_\pi\Vert\pi(\varphi)\Vert$$
where $\pi$ ranges over all bounded $*$-representations of  $\Cc_c(\Gc)$. 
One can similarly define the reduced $C^*$-algebra $C^*_{\rm red}(\Gc)$ 
by restricting the above supremum to the family of \emph{regular} representations 
$\{\Lambda_x\mid x\in \Gc^{(0)}\}$, where we define 
$$(\forall x\in \Gc^{(0)})\quad \Lambda_x\colon \Cc_c(\Gc)\to \Bc(L^2(\Gc,\lambda_x)),\quad \Lambda_x(\varphi)\psi:=\varphi\ast \psi, $$
where $\lambda_x$ is the pushforward of $\lambda^x$ under the inversion map on $\Gc$. 
There is a canonical surjective $*$-homomorphism $C^*(\Gc) \to C^*_{\rm red}(\Gc)$,  
and if it is also injective then the groupoid~$\Gc$ is called \emph{metrically amenable}.

\begin{remark}\label{rem2.1}
\normalfont 
Any locally closed subset (i.e., a difference of two open subsets) of a locally compact space is in turn locally compact 
with its relative topology. 
Then any $\Gc$-invariant locally closed subset $A\subseteq \Gc^{(0)} $ gives rise to a locally compact groupoid~$\Gc_A$,  
its corresponding \emph{reduced} groupoid, with its set of objects~$A$ 
and its set of morphisms $\Gc_A:=d^{-1}(A)$. 
We note that $\Gc_A=d^{-1}(A)\cap r^{-1}(A)=\{g\in\Gc\mid d(g)\in A, r(g)\in A\}$  since $A$ was assumed to be $\Gc$-invariant. 
Then $\Gc_A$ has a left Haar system $\lambda_A$ obtained by restricting the Haar system $\lambda$ of $\Gc$ to~$\Gc_A$,  
since the Tietze extension theorem implies that every function in $\Cc_c(\Gc_A)$ extends to a function in $\Cc_c(\Gc)$, 
using the fact that $\Gc_A$ is a locally closed subset of $\Gc$.   
In particular, we can construct the corresponding 
$C^*$-algebra $C^*(\Gc_A)$ and reduced $C^*$-algebra $C^*_{\rm red}(\Gc_A)$.  

If the above subset subset $A\subseteq\Gc^{(0)}$ is closed,   
then the subset $\Gc_A\subseteq\Gc$ is also closed.
Thus the restriction map $\Cc_c(\Gc)\to \Cc_c(\Gc_A)$  
is well defined, and it extends by continuity both to a $*$-homomorphism 
$\Rc_A\colon C^*(\Gc)\to C^*(\Gc_A)$ 
and to a $*$-homomorphism of the reduced $C^*$-algebras
$(\Rc_A)_{\rm red}\colon C^*_{\rm red}(\Gc)\to C^*_{\rm red}(\Gc_A)$ 
which are related by the commutative diagram 
\begin{equation}\label{fullred}
\begin{CD}
C^*(\Gc) @>{\Rc_A}>> C^*(\Gc_A) \\
@VVV @VVV \\
C^*_{\rm red}(\Gc) @>{(\Rc_A)_{\rm red}}>> C^*_{\rm red}(\Gc_A)
\end{CD}
\end{equation}
where the vertical arrows are the natural quotient homomorphisms. 
\end{remark}

For later use, we record the following basic facts on reductions of a groupoid to open or closed invariant subsets of the space of units. 

\begin{proposition}\label{Renault_page101}
If $\Gc$ is any locally compact groupoid with a left Haar system, 
then the following assertions hold: 
\begin{enumerate}[(i)]
\item\label{Renault_page101_item1} 
There exists a bijective correspondence $U\longleftrightarrow I(U)$ 
between $\Gc$-invariant open subsets $U\subseteq G^{(0)}$ and some closed 
two-sided ideals of $C^*(\Gc)$ 
such that for every~$U$ with its complement $F:=G^{(0)}\setminus U$ 
one has a short exact sequence 
$$0\to I(U)\to C^*(\Gc)\mathop{\longrightarrow}\limits^{\Rc_F} C^*(\Gc_F)\to 0$$
and a natural $*$-isomorphism $I(U)\simeq C^*(\Gc_U)$. 
\item\label{Renault_page101_item2} 
For every $U$ and $F$ as above one has the short 
exact sequence of multiplier algebras 
$$0\to \Ker\Rc^{**}_{F}
\to M(C^*(\Gc)) 
\mathop{\longrightarrow}\limits^{\Rc^{**}_F} 
M(C^*(\Gc_F))\to 0.$$
\item\label{Renault_page101_item3} 
If $U$ and $F$ as above have the additional property that 
the canonical quotient map $C^*(\Gc_F)\to C^*_{\rm red}(\Gc_F)$ is an isomorphism, 
then one has the short 
exact sequence of reduced $C^*$-algebras  
$$0\to C^*_{\rm red}(\Gc_U)\to C^*_{\rm red}(\Gc)
\mathop{\xrightarrow{\hspace*{1cm}}}\limits^{(\Rc_F)_{\rm red}} 
C^*_{\rm red}(\Gc_F)\to 0.$$
\end{enumerate}
\end{proposition}

\begin{proof}
The first assertion is well known, see for instance \cite[Lemma 2.1]{MRW96}. 

For the second assertion first recall from \cite[Subsect. 1.5 and Rem. 2.2.3]{We93} 
that if $\Ac$ is any $C^*$-algebra, 
then its multiplier algebra can be identified as a $C^*$-subalgebra of the universal enveloping von Neumann algebra 
$\Ac^{**}$ as $M(\Ac)=\{a\in\Ac^{**}\mid a\Ac+\Ac a\subseteq\Ac\}$. 
Then the conclusion follows by \cite[Th. 2.3.9]{We93}, 
which says  
that if $\Ac_1\to\Ac_2\to0$ is an exact sequence of $\sigma$-unital $C^*$-algebras, 
then the corresponding sequence of multiplier algebras $M(\Ac_1)\to M(\Ac_2)\to0$ 
is also exact, 
and every separable $C^*$-algebra is $\sigma$-unital. 
In particular, this is the case for 
$C^*_{\rm red}(\Gc)$, which is separable as a quotient of the separable $C^*$-algebra $C^*(\Gc)$. 

The third assertion was noted in \cite[Rem. 4.1]{Re91}.
\end{proof}

\begin{remark}\label{transitive}
\normalfont
Let $\Gc$ be a locally compact groupoid with a Haar system. 
Fix any $x\in\Gc^{(0)}$ for which the corresponding $\Gc$-orbit~$U:=\Gc\cdot x\subseteq\Gc^{(0)}$ is a locally closed set. 
Then there exist a positive measure $\mu$ on $U$ and a $*$-isomorphism 
$C^*(\Gc_U)\simeq C^*(\Gc(x))\otimes \Kc(L^2(U,\mu))$. 
(See \cite[Th. 3.1]{MRW87} and \cite[Th. 7]{Bu03}.) 
\end{remark}

Now we can draw the following corollary of Proposition~\ref{Renault_page101}.

\begin{corollary}
If $\Gc$ is a locally compact groupoid with a left Haar system 
for which there exists an open principal orbit $U\subseteq \Gc^{(0)}$, 
then there exists a positive measure $\mu$ on $U$ such that $C^*(\Gc)$ has a closed ideal isomorphic to 
$\Kc(L^2(U,\mu))$ and there exists a short exact sequence 
$$0\to \Kc(L^2(U,\mu))\to C^*(\Gc)\mathop{\longrightarrow}\limits^{\Rc_Y} C^*(\Gc_F)\to 0$$ 
where $F=\Gc^{(0)}\setminus U$.
\end{corollary}

\begin{proof}
We first note that the hypothesis that $U$ is a principal orbit is equivalent to 
the fact that $\Gc_U$ is a pair groupoid, and in particular $\Gc(x)=\{\1\}$ for any $x\in U$. 
Then Remark~\ref{transitive} implies $C^*(\Gc_U)\simeq \Kc(L^2(U,\mu))$ for a suitable measure $\mu$ on $U$. 
Now the conclusion follows by Proposition~\ref{Renault_page101}\eqref{Renault_page101_item1}. 
\end{proof}

\subsection{Open dense orbits} 

We now note a property of the ideals $I(U)$ from Proposition~\ref{Renault_page101}.
This is a very special instance of \cite[Th.~6.1]{CAR13}, with a stronger conclusion. 
Our result is also closely related to the interesting characterization of \'etale groupoids 
having a simple $C^*$-algebra, given in \cite[Th.~5.1]{BCFS14}. 
Here we denote by $\Kc$ the $C^*$-algebra of compact operators on a suitable separable complex Hilbert space. 

\begin{proposition}\label{fred0}
Let $\Gc\tto\Gc^{(0)}$ be any locally compact groupoid, which has  a Haar system, 
and whose orbits are locally closed.
 Then $C^*(\Gc) \simeq \Kc$ if and only if $\Gc$ is a pair groupoid. 
\end{proposition}

\begin{proof}
If $\Gc$ is a pair groupoid, then clearly  $C^*(\Gc) \simeq \Kc$ by Remark~\ref{transitive}.

Assume now that  $C^*(\Gc) \simeq \Kc$. 
We first prove that $\Gc$ is transitive. 
The orbit space $\Gc^{(0)}/\Gc$ is a $T_0$ space by \cite[Th. 2.1($(4)\Leftrightarrow(5)$)]{Ra90}. 
If there exist two distinct points $\Oc_1$ and $\Oc_2$ in 
$\Gc^{(0)}/\Gc$, then by the $T_0$ property, 
we may assume that there is an open neighbourhood  $V\subset \Gc^{(0)}/\Gc$ of $\Oc_1$ with $\Oc_2 \not\in V$. 
Denote by $q\colon \Gc^{(0)} \to \Gc^{(0)}/\Gc$ the quotient map. 
Then $U=q^{-1}(V)$ is a non-empty open subset on $\Gc^{(0)}$, different from $\Gc^{(0)}$,  and $\Gc$-invariant. 
It follows, by the bijective correspondence in Proposition~\ref{Renault_page101} \eqref{Renault_page101_item1}, 
that $I(U)$ is a nontrivial closed ideal of $C^*(\Gc)$. 
This is impossible since $C^*(\Gc) \simeq \Kc$. 
Therefore $\Gc$ is transitive. 

By Remark~\ref{transitive} we have now that, for any $x \in \Gc^{(0)}$, 
$C^*(\Gc) \simeq C^*(\Gc (x)) \otimes \Kc$. 
If $\Gc (x)\ne \{1\}$, then by the  Gelfand-Raikov Theorem there is 
an irreducible representation $\pi$ of $\Gc(x)$
different from the trivial representation $\tau$.
Thus $\pi\otimes \text{id}$ and $\tau\otimes \text{id}$ are two non-equivalent irreducible representations of 
$C^*(\Gc(x))\otimes\Kc$, which is a contradiction with the assumption $C^*(\Gc) \simeq \Kc$.  
Hence we must have $\Gc(x) =\{1\}$, and this finishes the proof. 
\end{proof}

\begin{lemma}\label{dense1}
If $\Gc\tto \Gc^{(0)}$  is any locally compact groupoid 
with a left Haar system.  
 Assume that for  $x_0\in \Gc^{(0)}$
the regular representation 
$\Lambda_{x_0}\colon C^*(\Gc) \to \Lc(L^2(\Gc_{x_0}))$ is injective. 
Then the $\Gc$-orbit of $x_0$ is dense in $\Gc^{(0)}$. 
\end{lemma}

\begin{proof}
We will actually prove a stronger fact, 
namely that if the regular representation $\Lambda_{x_0}\vert_{\Cc_c(\Gc)}$ is injective, 
then the orbit $\Gc\cdot x_0$ is dense in~$\Gc^{(0)}$. 
To this end we assume that there exists 
an open nonempty set $U\subseteq \Gc^{(0)}$ with $U\cap\Gc\cdot x_0=\emptyset$, 
and we will show that this leads to a contradiction. 
Specifically, since $U$ is open and nonempty, then $V:=r^{-1}(U)$ ($\supseteq U$) is an open nonempty 
subset of $\Gc$. 
Since $\Gc$ is locally compact, it then easily follows by Urysohn's lemma that there exists 
$\varphi\in\Cc_c(\Gc)\setminus\{0\}$ with $\supp\varphi\subseteq V$. 
In particular, $\varphi(k)=0$ if $k\in\Gc\setminus V=r^{-1}(\Gc^{(0)}\setminus U)\supseteq r^{-1}(\Gc\cdot x_0)$. 
Then for every $\psi\in\Cc_c(\Gc_{x_0})$ and $g\in\Gc_{x_0}$ we have 
 $$(\Lambda_{x_0}(\varphi)\psi)(g)=\int\limits_{\Gc_{x_0}}\varphi(h)\psi(h^{-1}g) d h=0 $$
 because here we have $r(h)=r(g)\in r(\Gc_{x_0})=\Gc\cdot x_0$, hence $\varphi(h)=0$. 
Since $\Cc_c(\Gc_{x_0})$ is dense in $L^2(\Gc_{x_0})$, it then follows $\varphi\in\Ker(\Lambda_{x_0}\vert_{\Cc_c(\Gc)})\setminus\{0\}$, 
which is a contradiction with our assumption.  
\end{proof}

\begin{proposition}\label{dense2}
Let $\Gc\tto \Gc^{(0)}$ be any locally compact groupoid 
with a left Haar system. 
If $U\subseteq \Gc^{(0)}$ is any open $\Gc$-invariant set and $x_0\in U$, then the following assertions hold: 
\begin{enumerate}[(i)]
\item\label{dense2_item1} 
For every $x\in \Gc^{(0)}\setminus U$ 
the ideal $C^*(\Gc_U)$ of $C^*(\Gc)$ is contained in the kernel of the regular representation 
$\Lambda_x\colon C^*(\Gc)\to\Lc(L^2(\Gc_x))$, 
and similarly,  for the reduced $C^*$-algebras, the ideal $C^*_{\rm red}(\Gc_U)$ of $C^*_{\rm red}(\Gc)$ 
is contained in the kernel of the regular representation 
$\Lambda_x\colon C^*_{\rm red}(\Gc)\to\Lc(L^2(\Gc_x))$. 
\item\label{dense2_item2} 
If $U$ is an orbit of $\Gc$, then the regular representation 
$\Lambda_{x_0}$  
is faithful if and only if $U$ is dense in~$\Gc^{(0)}$.  
\end{enumerate}
\end{proposition}

\begin{proof}
Assertion~\eqref{dense2_item1}  directly follows as a byproduct of 
the proof of Lemma~\ref{dense1}, where we actually checked that if $U\cap\Gc\cdot x_0=\emptyset$ and 
$\supp\varphi\subseteq r^{-1}(U)=\emptyset$ then $\varphi\in\Ker\Lambda_{x_0}$.

For Assertion~\eqref{dense2_item2}, first note that since $\Gc_U$ is transitive and $x_0\in U$, 
it follows that $U=\Gc.x_0$. 
Now, if the representation $\Lambda_{x_0}$ is faithful, 
then the set $U$ is dense again as a consequence of Lemma~\ref{dense1}. 
 Conversely, assume that $U$ is dense in $M$. 
Since the topology of $U$ is second countable, 
we may select any sequence of points $x_1,x_2,\dots\in U$ which is dense in $U$. 
The corresponding infinite convex combination of Dirac measures $\nu:=\sum\limits_{n\ge1}\frac{1}{2^n}\delta_{x_n}$ 
is a measure on $\Gc^{(0)}$ with dense support, hence by \cite[Cor.~2.4]{KS02} 
the representation $\Ind_\nu:={\int\limits_{\Gc^{(0)}}}^\oplus\Lambda_x\de\nu(x)$ of $C^*_{\rm red}(\Gc)$ is faithful. 
But for every $n\ge1$ the representation $\Lambda_{x_n}$ is unitarily equivalent to $\Lambda_{x_0}$ 
because $x_n\in U=\Gc.x_0$, hence it follows that $\Ker\Lambda_{x_0}=\Ker(\Ind_\nu)=\{0\}$, 
and this completes the proof. 
\end{proof}

\section{Group bundles}\label{Sect3}

In this section we study group bundles, a special type of groupoids  
that will turn out to play an important role in our investigation of transformation groups 
in the next section (see for instance Corollary~\ref{Morita4}). 

\begin{definition}
\normalfont
A \emph{group bundle} is a locally compact groupoid 
whose range and domain maps are equal. 
\end{definition} 

\begin{remark}\label{Re91_Lemma1.3}
\normalfont 
We recall that by definition the range and domain maps of any locally compact groupoid are assumed to be open maps. 
Then, as a direct consequence of \cite[Lemma 1.3]{Re91}, any group bundle has a left Haar system.  
\end{remark}

We record the following result for further reference although it 
is known and can be traced back to \cite[Lemma 1.1A]{Gl62}. 

\begin{lemma}\label{1Cstar_isotrop}
Let $p\colon \Tc\to S$ be a group bundle with a fixed Haar system 
and denote $\Tc_s:=p^{-1}(s)$ for every $s\in S$.
The following assertions hold: 
\begin{enumerate}[(i)]
\item\label{1Cstar_isotrop_item1}
The $C^*$-algebra $C^*(\Tc)$ is a $\Cc_0(S)$-algebra 
and is $\Cc_0(S)$-linearly $*$-iso\-morphic   
to the algebra of sections of an upper semicontinuous $C^*$-bundle over~$ S$ 
whose fiber over any $s\in S$ is $C^*(\Tc_s)$. 
\item\label{1Cstar_isotrop_item2}  
One has a partition into closed subsets 
$\widehat{C^*(\Tc)}=\bigsqcup\limits_{s\in S}\widehat{C^*(\Tc_s)}$
and there is a continuous map $\pi\colon \widehat{C^*(\Tc)}\to S$ 
with $\pi^{-1}(s)=\widehat{C^*(\Tc_s)}$ for every $s\in S$.
\end{enumerate}
\end{lemma}

\begin{proof}
For Assertion~\eqref{1Cstar_isotrop_item1}, 
we need to construct a nondegenerate $*$-morphism $\Phi$ from $\Cc_0(S)$ into 
the center $ZM(C^*(\Tc))$ of the multiplier algebra $M(C^*(\Tc))$ of $C^*(\Tc)$ 
(see \cite[Def. C.1]{Wi07}). 
Here we regard $M(C^*(\Tc))$ as the $C^*$-algebra of adjointable operators on 
$C^*(\Tc)$ viewed as a right Hilbert module over itself 
with the scalar multiplication and inner product defined by $\varphi\cdot\psi:=\varphi\ast\psi$ 
and $\langle\varphi,\psi\rangle:=\varphi^*\ast\psi$. 
For all $\varphi,\psi\in\Cc_c(\Tc)$ their product in $C^*(\Tc)$ is 
the function $\varphi\ast\psi\in\Cc_c(\Tc)$ given by 
$$(\varphi\ast\psi)(g)=\int_{\Tc_s}\varphi(h)\psi(h^{-1}g)\de\lambda^{s}(h)
\text{ for }g\in\Tc_s\text{ and }s\in S.$$ 
Let $f\in\Cc_0(S)$ be arbitrary and define $\Phi_0(f)\colon \Cc_c(T)\to\Cc_c(T)$, $\Phi_0(f)\varphi:=(f\circ p)\cdot\varphi$.  
It is straightforward to check that 
$$\langle\Phi_0(f)\varphi,\psi\rangle=\langle \varphi,\Phi_0(\bar f)\psi\rangle 
\text{ and }\Phi_0(f)\varphi \ast \psi=\varphi \ast \Phi_0(f)\psi=\Phi_0(f)(\varphi\ast\psi)$$ 
which imply at once that 
$\langle\Phi_0(f)\varphi,\Phi_0(f)\varphi\rangle\le\Vert f\Vert_\infty^2\langle\varphi,\varphi\rangle$, 
hence $\Phi_0(f)$ extends to a bounded linear operator $\Phi(f)\colon C^*(\Tc)\to C^*(\Tc)$. 
Moreover $\Phi(f)\in ZM(C^*(\Tc))$ by \cite[Lemma 8.3]{Wi07}. 
One thus obtains a $*$-morphism $\Phi\colon \Cc_0(S)\to ZM(C^*(\Tc))$, 
and it is easily seen that the set $\{\Phi_0(f)\varphi\mid f\in\Cc_0(S),\varphi\in\Cc_0(\Tc)\}$ is dense in $C^*(\Tc)$. 
Hence the morphism $\Phi$ is nondegenerate, 
and it thus defines a $\Cc_0(S)$-algebra structure on $C^*(\Tc)$. 
The second part of the assertion then follows by \cite[Th. C.26]{Wi07}.  
Assertion~\eqref{1Cstar_isotrop_item2} follows by \cite[Prop. C.5]{Wi07}.  
\end{proof}

\begin{lemma}\label{bund2}
Let $p\colon \Tc\to S$ be any group bundle and $\theta\colon\Xi\to S$ be any continuous map, 
and define 
$\theta^*(\Tc):=\{(\xi,t)\in\Xi\times \Tc\mid \theta(\xi)=p(t)\}\subseteq \Xi\times \Tc$
with its relative topology. 
Then $q\colon \theta^*(\Tc)\to\Xi$, $(\xi,t)\mapsto\xi$, has the canonical structure of a group bundle. 
\end{lemma}

\begin{proof}
It is clear that the fibers of $q\colon \theta^*(\Tc)\tto\Xi$  
are locally compact groups, hence it remains to check that $q$ is an open map. 
To this end, for arbitrary open subsets $V\subseteq \Xi$ and $W\subseteq \Tc$ with $(V\times W)\cap\theta^*(p)\ne\emptyset$, 
we must prove that $q((V\times W)\cap\theta^*(p))$ is an open subset of $\Xi$. 
One has 
$$q((V\times W)\cap\theta^*(p))=\{\xi\in\Xi\mid (\exists t\in W)\ \theta(\xi)=p(t)\}=\theta^{-1}(p(W))$$
and this is an open subset of $\Xi$ since $\theta$ is continuous and $p$ is an open map. 
\end{proof}

A \emph{strong group bundle} is a group bundle $s\colon\Tc\to S$ 
with a fixed Haar system (see Remark~\ref{Re91_Lemma1.3})  
for which $C^*(\Tc)$ is $\Cc_0(S)$-linearly $*$-iso\-morphic   
to the $C^*$-algebra of sections of a continuous $C^*$-bundle $q\colon E\to S$. 
We note that by Lemma~\ref{1Cstar_isotrop}\eqref{1Cstar_isotrop_item1} there always exists such an upper semicontinuous $C^*$-bundle.

\begin{lemma}\label{bund2bis}
In the setting of Lemma~\ref{bund2},  
the following assertions hold: 
\begin{enumerate}[(i)]
\item\label{bund2bis_item1}
\begin{enumerate}
\item\label{bund2bis_item1a} 
If $p$ is a strong group bundle, then for every $\psi\in\Cc_c(\Tc)$ the function 
$\widetilde\psi\colon S\to[0,\infty)$, $\widetilde\psi(s):=\Vert\psi\vert_{\Tc_s}\Vert_{C^*(\Tc_s)}$ is continuous. 
\item\label{bund2bis_item1b} 
Conversely, if there is a $*$-subalgebra $\Fc$ of the convolution algebra $\Cc_c(\Tc)$ 
such that for every $\psi\in\Fc$ its corresponding function 
$\widetilde\psi\colon S\to[0,\infty)$ is continuous, 
and for every $s\in\Sc$ the set $\{\psi\vert_{\Tc_s}\mid \psi\in\Fc\}$ is dense in $C^*(\Tc_s)$, 
then $p\colon\Tc\to S$ is a strong group bundle. 
\end{enumerate}
\item\label{bund2bis_item2}
If  $p\colon\Tc\to S$ is a strong group bundle, 
then also $q\colon \theta^*(\Tc)\to\Xi$ is a strong group bundle and one has 
$\Cc_0(\Xi)\otimes_{\Cc_0(S)} C^*(\Tc)\simeq C^*(\theta^*(\Tc))$ 
via a $\Cc_0(\Xi)$-linear isometric $*$-isomorphism  induced by the bilinear map 
$\Cc_c(\Xi)\times \Cc_c(\Tc)\to \Cc_c(\theta^*(\Tc))$, 
$(f,\psi)\mapsto(f\otimes \psi)\vert_{\theta^*(\Tc)}$. 
\end{enumerate}
\end{lemma}

\begin{proof}
Let $E:=\{C^*(\Tc_s)\}_{s\in S}$ be the upper semicontinuous $C^*$-bundle 
whose $C^*$-algebra of global sections $\Gamma_0(E)$ 
is canonically isomorphic to the groupoid $C^*$-algebra $C^*(\Tc)$. 

For every $\psi\in\Cc(\Tc)$, its corresponding map $s\mapsto \psi\vert_{\Tc_s}\in C^*(\Tc_s)$ 
is a global cross-section of $E$,  
hence both parts \eqref{bund2bis_item1}--\eqref{bund2bis_item2} of Assertion~\eqref{bund2bis_item1} 
follow by Fell's theorem \cite[Th. C.25]{Wi07}. 

For Assertion~\eqref{bund2bis_item2}, assume that $p\colon\Tc\to S$ is a strong group bundle, 
hence $q\colon \theta^*(\Tc)\to\Xi$ is a group bundle by Lemma~\ref{bund2}. 
Then we note that the property that $q$ be a strong bundle is a local property with respect 
to the points in~$\Xi$. 
Therefore, restricting $q$ to compact subsets of $\Xi$ with nonempty interior, 
we may assume that $\Xi$ is compact. 

Then let $\psi\in\Cc_c(\Tc)$ arbitrary, and  
define $\1\otimes\psi\colon \theta^*(\Tc)\to\CC$, 
$(\1\otimes\psi)(\xi,t)=\psi(t)$ for all $(\xi,t)\in\theta^*(\Tc)$. 
It is clear that $\1\otimes\psi\in\Cc(\theta^*(\Tc))$. 
We now check that the set $\Fc:=\spann\{\1\otimes\psi\mid \psi\in\Cc_c(\Tc)\}$ 
satisfies the hypothesis of Assertion~\eqref{bund2bis_item2} for the group bundle $q\colon\theta^*(\Tc)\to\Xi$. 
If  $\psi_1,\psi_2\in\Cc_c(\Tc)$, 
then  
$(\1\otimes\psi_1)\ast(\1\otimes\psi_2)=\1\otimes(\psi_1\ast\psi_2)$. 
One also easily checks that $(\1\otimes\psi)^*=\1\otimes\psi^*$, and it then follows that 
$\Fc$ is a $*$-subalgebra of the convolution algebra $\Cc_c(\theta^*(\Tc))$. 

For every $\xi\in\Xi$ one has 
$(\1\otimes\psi)\vert_{(\theta^*(\Tc))_\xi}=\psi\vert_{\Tc_{\theta(\xi)}}\in\Cc_c(\Tc_{\theta(\xi)})$, 
which implies at once that $\{(\1\otimes\psi)\vert_{(\theta^*(\Tc))_\xi}\mid\psi\in\Cc(\Tc)\}$ 
contains $\Cc_c(\Tc_{\theta(\xi)})$, hence is dense in $C^*(\Tc_{\theta(\xi)})=C^*((\theta^*(\Tc))_\xi)$. 
Also, one has $\Vert(\1\otimes\psi)\vert_{(\theta^*(\Tc))_\xi}\Vert_{C^*((\theta^*(\Tc))_\xi)}
=\Vert\psi\vert_{\Tc_{\theta(\xi)}}\Vert_{C^*(\Tc_{\theta(\xi)})}$, 
which depends continuously on $\xi\in\Xi$ since $\theta\colon\Xi\to S$ is continuous by hypothesis 
and $s\mapsto\Vert\psi\vert_{\Tc_s}\Vert_{C^*(\Tc_s)}$ is continuous by Assertion~\eqref{bund2bis_item1a}. 

Thus the group bundle $q\colon\theta^*(\Tc)\to\Xi$ satisfies the hypothesis of Assertion~\eqref{bund2bis_item1b},  
and  then $q$ is a strong group bundle. 
Finally, the $*$-isomorphism from the statement follows by \cite[Prop. 1.3]{RWi85}, 
and we are done. 
\end{proof}

Now we show that to every locally compact group there corresponds a group bundle 
whose fibers are the closed subgroups of the group under consideration. 
We establish in Proposition~\ref{strong} a lesser known aspect of 
this construction that goes back to \cite{Fe64} and is a version of the construction of the tautological vector bundle 
over the Grassmann manifold of subspaces of a vector space. 
It is remarkable that the well known universality property of the tautological vector bundle 
has an analogue for the group bundles associated to locally compact groups and 
this idea plays a central role in the investigation of transformation groups in terms of their isotropy groups 
(see particularly Corollary~\ref{Morita4}).
Specifically, for any locally compact group~$G$ one defines 
\begin{itemize}
\item $\Kc(G):=\{K\mid K\text{ closed subgroup}\subseteq G\}$ 
\item $\Tc(G):=\{(K,g)\in\Kc(G)\times G\mid g\in K\}$ 
\item $p\colon\Tc(G)\to\Kc(G)$, $p(K,g):=K$.
\end{itemize}
Then it is well known that $\Kc(G)$ is a compact topological space with its Fell topology, for which a basis   consists
of the sets 
\begin{equation}
\label{Fellbasis}
\Uc(C,\Sc):=\{K\in\Fc(G)\mid K\cap C=\emptyset;\ (\forall A\in\Sc)\ K\cap A\ne\emptyset\}
\end{equation}
for all compact sets $C\subseteq G$ and all finite sets $\Sc$ of open subsets of~$H$. 
Moreover $\Tc(G)$ is a  closed subset of $\Kc(G)\times G$, 
hence $\Tc(G)$ is locally compact with its relative topology. 
(See for instance \cite[App. H]{Wi07}.)
We will prove in Lemma~\ref{bund1} below that $p\colon\Tc(G)\to\Kc(G)$ is a group bundle, 
hence by Remark~\ref{Re91_Lemma1.3} it has Haar systems (see also \cite[Appendix]{Gl62}). 
We fix a Haar system on the group bundle $p\colon\Tc(G)\to\Kc(G)$  
and we denote its corresponding $C^*$-algebra by $C^*(\Tc(G))$. 

\begin{lemma}\label{bund1}
For any locally compact group $G$ the map $p\colon \Tc(G)\to\Kc(G)$ 
is a group bundle.  
\end{lemma}

\begin{proof} 
We keep the notation in \eqref{Fellbasis} above. 
We must prove that for any set $\Uc(C,\Sc)$ as above and any open set $D\subseteq G$, 
if $(\Uc(C,\Sc)\times D)\cap \Tc(G)\ne\emptyset$, then $p((\Uc(C,\Sc)\times D)\cap \Tc(G))$ is an open subset of $\Kc(G)$. 
Indeed, for any $K\in\Fc(G)$, we have $K\in p((\Uc(C,\Sc)\times D)\cap \Tc(G))$ if and only if $K\in \Uc(C,\Sc)$ and 
there exists $x\in D\cap K$, 
since if that is the case, then $(K,x)\in (\Uc(C,\Sc)\times D)\cap \Tc(G)$ and $K=p(K,x)$. 
Therefore 
$$p((\Uc(C,\Sc)\times D)\cap \Tc(G))=\{K\in\Uc(C,\Sc)\mid K\cap D\ne\emptyset\}=\Uc(C,\Sc\cup\{D\})$$
which is an open subset of $\Kc(G)$, and this concludes the proof. 
\end{proof}

It is easily seen that the proof of Lemma~\ref{bund1} carries over for the space of closed subsets of any locally compact space 
(rather than closed subgroups of a locally compact group), and leads to the fact that the corresponding map $(K,g)\mapsto K$ 
is open. 

\begin{proposition}\label{strong}
If a locally compact group $G$ is amenable, then $p\colon\Tc(G)\to\Kc(G)$ is a strong group bundle. 
\end{proposition}

\begin{proof}
Since $p\colon\Tc(G)\to\Kc(G)$ is a group bundle, its $C^*$-algebra 
$C^*(\Tc(G))$ has the canonical structure of a $\Cc(\Kc(G))$-algebra 
and is $\Cc(\Kc(G))$-linearly $*$-iso\-morphic   
to the algebra of sections of an upper semicontinuous $C^*$-bundle $q\colon \Ec(G)\to \Kc(G)$ 
whose fiber over any $K\in\Kc(G)$ is $*$-isomorphic to $C^*(K)$,  
by Lemma~\ref{1Cstar_isotrop}\eqref{1Cstar_isotrop_item1}. 
We must prove that $q\colon \Ec(G)\to \Kc(G)$ is a continuous $C^*$-bundle. 
That continuity property will be obtained by an application of \cite[Th. C.26]{Wi07}, 
and for that we need to check that the continuous map 
$\pi\colon \widehat{C^*(\Tc(G))}\to\Kc(G)$ which defines 
the $\Cc(\Kc(G))$-algebra structure of $C^*(\Tc(G))$ is also an open map. 

To this end we will first describe the map $\pi$ in more detail. 
It follows by \cite[Lemmas 2.6 and 2.8]{Fe64} that for every $K\in\Kc(G)$ 
one has a canonical injective map $\widehat{K}\hookrightarrow\widehat{C^*(\Tc(G))}$ which is a homeomorphism 
onto its image, and whose image is closed, 
and moreover 
$\widehat{C^*(\Tc(G))}=\bigsqcup\limits_{K\in\Kc(G)}\widehat{K}$. 
The natural map $\widehat{C^*(\Tc(G))}\to\Kc(G)$, $x\mapsto K$ if $x\in\widehat{K}$,  defined by the above partition, 
 is continuous by \cite[Lemmas 2.8 and 2.5]{Fe64} and is exactly the map $\pi$ we seek to describe. 

Since the group $G$ is  amenable, it follows that every $K\in\Kc(G)$ is an amenable locally compact group, 
hence every unitary irreducible representation of $K$ is weakly contained in the regular representation of~$K$. 
It then follows by \cite[Th. 2.1 and Cor. 6 in \S 5]{Fe64} that $\pi\colon\widehat{C^*(\Tc(G))}\to\Kc(G)$ is an open map, 
and we are done. 
\end{proof}

\subsection{Duals of $C^*$-algebras of regular groupoids}

In this subsection we discuss a class of groupoids  
for which the irreducible representations of their corresponding $C^*$-algebras 
can be described in terms of the representation theory of their isotropy groups. 

We will use the following terminology 
which is inspired by (but slightly different from) \cite{Go10} and \cite{Go12};  
see also \cite{Ra90}.

\begin{definition}\label{regular}
\normalfont
A locally compact groupoid $\Gc\tto\Gc^{(0)}$ endowed with a left Haar system~$\lambda$  
is called \emph{regular} if it satisfies the following additional conditions: 
\begin{enumerate}
\item Every orbit of $\Gc$ is a locally closed subset of $\Gc^{(0)}$.
\item The isotropy subgroupoid $\Gc(\cdot):=\bigsqcup\limits_{x\in \Gc^{(0)}}\Gc(x)$ 
is a group bundle. 
\end{enumerate}
\end{definition}
In the above definition, the condition that $\Gc(\cdot)$ be a group bundle actually requires that 
the natural projection $\Gc(\cdot)\to\Gc^{(0)}$ is an open map, 
and this implies that $\Gc(\cdot)$ is a locally compact groupoid with a Haar system (see Remark~\ref{Re91_Lemma1.3}). 
Some examples of regular groupoids are constructed in Proposition~\ref{29Nov2015} 
as groupoid pullbacks of group bundles. 

With Lemma~\ref{1Cstar_isotrop} at hand, we can prove the following theorem, 
where we use the notation $\Ind$ for induced representations of groupoids 
(see for instance \cite{IW09a}). 
This theorem goes back to \cite[Th. 2.1]{Gl62} in the special case of transformation groups. 

\begin{theorem}\label{Cstar_grpd}
For a regular groupoid $\Gc\tto\Gc^{(0)}$,  
define 
$$\Phi\colon \widehat{C^*(\Gc(\cdot))} 
\to\widehat{C^*(\Gc)}, 
\quad 
\Phi([\pi]):=[\Ind_{\Gc(x)}^{\Gc}\pi],$$
for $[\pi]\in \widehat{\Gc(x)}\simeq\widehat{C^*(\Gc(x))}\subseteq\widehat{C^*(\Gc(\cdot))}$ 
for $x\in\Gc^{(0)}$. 
Then $\Phi$ is a continuous open surjective map, and it induces a homeomorphism 
$\widehat{C^*(\Gc(\cdot))}/\Gc\simeq\widehat{C^*(\Gc)}$, 
where the left-hand side is the quotient of 
$\widehat{C^*(\Gc(\cdot))}\simeq\bigsqcup\limits_{x\in\Gc^{(0)}}\widehat{\Gc(x)}$ 
by the natural action of~$\Gc$. 
\end{theorem}

\begin{proof}
The definition of $\Phi$ is correct by \cite[Prop. 4.13]{Go10}, 
where it was established 
that every  irreducible representation of some isotropy group of~$\Gc$ induces an 
irreducible representation of $C^*(\Gc)$.
The assertion follows by Lemma~\ref{1Cstar_isotrop} and \cite[Th. 2.22]{Go12}. 
\end{proof}

For a locally compact space $T$ we denote by $\beta T$ its Stone-\v Cech compactification.

\begin{corollary}\label{multipl}
Let $\Gc\tto\Gc^{(0)}$ be a regular groupoid
whose isotropy groups $\Gc(x)$ are  of type I and amenable. 
Assume in addition that the orbit space $\Gc^{(0)}/\Gc$  is locally compact and Hausdorff. 
Then $M(C^*(\Gc))$ is a $C(\beta(\Gc^{(0)}/\Gc))$-algebra.
\end{corollary}

\begin{proof}
	We first recall that \cite[Th.~7.2]{Cl07} ensures that if $\Gc$ is a groupoid in which all the isotropy groups are amenable, then $C^*(\Gc)$ is GCR if and only if the orbit space is $T_0$ and the isotropy groups are GCR.  Since $\Gc(x)$ are type I and have separable $C^*$-algebras, we get that $C^*(\Gc)$ is of type I.
From Lemma~\ref{1Cstar_isotrop} we have that $C^*(\Gc(\cdot))$ is a $C_0(\Gc^{(0)})$-algebra, hence
there is a continuous map $\widehat {C^*(\Gc(\cdot))}\to \Gc^{(0)}$.  
This map commutes with the natural actions of the groupoid $\Gc$ on its domain and  range, hence it induces
a continuous map $\widehat {C^*(\Gc(\cdot))}/\Gc \to \Gc^{(0)}/\Gc$.
By Theorem~\ref{Cstar_grpd} we obtain a continuous map $\widehat {C^*(\Gc)} \to \Gc^{(0)}/\Gc$.
This implies, by \cite[Prop.~1.2]{AS11}, that $M(C^*(\Gc))$ is a $C(\beta(\Gc^{(0)}/\Gc))$-algebra.
This concludes the proof.
\end{proof}

\section{Pullback of groupoids}\label{Sect4}

A general method of constructing new examples of groupoids is the pullback (see for instance \cite[\S~3.3]{Wi15}).
As we will see below, the operation of pullback by continuous open surjections  
 preserves most of the properties that are relevant from the operator algebraic perspective, 
as for instance the topological properties of the orbits, the homeomorphism class of the orbit space, the isomorphism classes of isotropy groups, 
and existence of Haar systems. 
For the sake of completeness, we give  below a detailed presentation of an important result (Theorem~\ref{Morita1}), that asserts that the pullback of a groupoid by a continuous open surjection is equivalent to the initial groupoid.
This will be used in the later sections via its Corollary~\ref{Morita4}.

\begin{definition}
\normalfont
Let $\Gc\tto \Gc^{(0)}$ be a groupoid and $\theta\colon N\to \Gc^{(0)}$ be any map. 
The \emph{pullback} of $\Gc$ by $\theta$ is the groupoid $\theta^{\pullback}(\Gc)\tto N$ 
defined by 
$$\theta^{\pullback}(\Gc):=\{(n_2,g,n_1)\in N\times\Gc\times N\mid g\in\Gc_{\theta(n_1)}^{\theta(n_2)}\} $$ 
with its domain/range maps defined by $d(n_2,g,n_1)=n_1$ and $r(n_2,g,n_1)=n_2$ for all $(n_2,g,n_1)\in\theta^{\pullback}(\Gc)$. 
 The middle projection $\Theta(n_2, g, n_1)= g$, defines a map $\Theta\colon \theta^{\pullback}(\Gc)\to \Gc$ which is a groupoid morphism. 

If $\Gc\tto \Gc^{(0)}$ is a topological groupoid, $N$ a topological space and $\theta$ is continuous, 
then $\theta^{\pullback}(\Gc)\tto N$ is a topological groupoid 
with its topology induced from $N\times\Gc\times N$. 
\end{definition} 

We show in Proposition~\ref{29Nov2015} below that pullbacks of group bundles by surjective open maps are regular groupoids.

\begin{remark}
\normalfont
The map $\Theta$   
makes the following diagram commutative (in which the vertical arrows are either both domain maps or both range maps)
$$\begin{CD}
\theta^{\pullback}(\Gc) @>{\Theta}>> \Gc \\
@VVV @VVV \\
N @>{\theta}>> \Gc^{(0)}
\end{CD}$$
and moreover the map $\Theta$ gives a bijection 
$(\theta^{\pullback}(\Gc))_{n_1}^{n_2}\to \Gc_{\theta(n_1)}^{\theta(n_2)}$, 
for all $n_1,n_2\in N$.  

In particular, for every $n_0\in N$, one has the algebraic isomorphism of isotropy groups 
$(\theta^{\pullback}(\Gc))_{n_0}^{n_0}\to \Gc_{\theta(n_0)}^{\theta(n_0)}$, 
which is also an isomorphism of topological groups if $\Gc\tto \Gc^{(0)}$ is a topological groupoid. 
\end{remark}

\begin{remark}
\normalfont
For every $n_0\in N$, its corresponding $d$-fiber of $\theta^{\pullback}(\Gc)$ can be described as 
\begin{equation}\label{pb_eq1}
(\theta^{\pullback}(\Gc))_{n_0}=\{(n,g,n_0)\in \theta^{\pullback}(\Gc)\mid g\in\Gc_{\theta(n_0)},\ r(g)=\theta(n)\}
\end{equation}
hence one has the commutative diagram 
$$\begin{CD}
(\theta^{\pullback}(\Gc))_{n_0} @>{\Theta}>> \Gc_{\theta(n_0)} \\
@V{r}VV @VV{r}V \\
N @>{\theta}>> \Gc^{(0)}
\end{CD}
$$
Thus $(\theta^{\pullback}(\Gc))_{n_0}$ is in fact  the fiber product $\Gc_{\theta(n_0)}\times_{\Gc^{(0)}} N$. 
\end{remark}

\begin{remark}
\normalfont
The $\theta^{\pullback}(\Gc)$-orbit of the point $n_0\in N$ is the inverse image through $\theta$ 
of the $\Gc$-orbit of the point $\theta(n_0)\in \Gc^{(0)}$, since, using \eqref{pb_eq1}, one obtains 
\begin{equation}\label{pb_eq2}
(\theta^{\pullback}(\Gc))\cdot n_0=r((\theta^{\pullback}(\Gc))_{n_0})=\theta^{-1}(r(\Gc_{\theta(n_0)}))=
\theta^{-1}(\Gc\cdot \theta(n_0)).
\end{equation}
This shows that if both $\Gc^{(0)}$ and $N$ are topological spaces, $\theta$ is continuous, 
and the $\Gc$-orbit of the point $\theta(n_0)$ is an open/closed/locally closed subset of $\Gc^{(0)}$, 
then so is the $\theta^{\pullback}(\Gc)$-orbit of the point $n_0$ in~$N$. 
\end{remark}

\begin{proposition}\label{beta}
The orbit spaces of the groupoids $\theta^{\pullback}(\Gc)$ and $\Gc$ are related by 
$$\beta\colon N/\theta^{\pullback}(\Gc)\to \Gc^{(0)}/\Gc,\quad (\theta^{\pullback}(\Gc)).n\mapsto \Gc.\theta(n)$$
which is a well-defined map and has the following properties:
\begin{enumerate}[(i)]
\item\label{beta_item1} $\beta$ is injective; 
\item\label{beta_item2} if the image of $\theta$ intersects every $\Gc$-orbit, then also $\beta$ is surjective; 
\item\label{beta_item3} if $\Gc$ is a topological groupoid and the map $\theta$ is continuous, 
then $\beta$ is continuous with respect to the quotient toplogies on the orbit spaces; 
\item\label{beta_item4} if $\Gc$ is a topological groupoid whose range map $r\colon \Gc\to \Gc^{(0)}$ is an open map, 
and the map 
$\theta$ is continuous, open, and surjective, 
then $\beta$ is a homeomorphism. 
\end{enumerate}
\end{proposition}

\begin{proof}
In fact, properties \eqref{beta_item1}--\eqref{beta_item2} follow by \eqref{pb_eq2}. 
For proving property~\eqref{beta_item3}, one needs the commutative diagram 
$$\begin{CD}
N @>{\theta}>> \Gc^{(0)} \\
@VVV @VVV \\
N/\theta^{\pullback}(\Gc) @>{\beta}>> \Gc^{(0)}/\Gc
\end{CD}$$
whose vertical arrows are the quotient maps, and in which the map $\theta$ is continuous. 
For property~\eqref{beta_item4} we also use the above diagram to check that if both 
$\theta$ and the quotient map $\Gc^{(0)}\to \Gc^{(0)}/\Gc$ are open, 
then $\beta$ is an open map, 
and then the assertion follows using also properties \eqref{beta_item1}--\eqref{beta_item3}. 
It remains to note that since the range map $r\colon \Gc\to \Gc^{(0)}$ is an open map, 
the quotient map $q\colon \Gc^{(0)}\to \Gc^{(0)}/\Gc$ is always open.
Indeed, this follows because for every open set $U\subseteq \Gc^{(0)}$ one has 
$q^{-1}(q(U))=r(d^{-1}(U))$ which is open in $\Gc^{(0)}$, hence $q(U)\subseteq \Gc^{(0)}/\Gc$  
is open by the definition of the quotient topology. 
\end{proof}

\subsection*{Morita equivalence}
In the following theorem we use the notion of equivalence of groupoids in the sense of \cite[Def. 2.1]{MRW87}. 
Results similar to this theorem 
were earlier established for instance in \cite[\S~3.3]{Wi15} and the references therein, but we provide   the 
full details of the proof here for the sake of completeness.

\begin{theorem}\label{Morita1}
Assume the following: 
\begin{enumerate}
\item  $\Gc\tto \Gc^{(0)}$ is a locally compact groupoid, 
\item $N$ is a second countable, locally compact topological space, 
\item 
$\theta\colon N\to \Gc^{(0)}$ is a continuous open surjective map.  
\end{enumerate}
Define 
the fibered product 
$$Z:=\Gc \times _{\Gc^{(0)}} N:=\{(g,n)\in \Gc\times N\mid d(g)=\theta(n)\}$$
and the maps $\rho\colon Z\to\Gc^{(0)}$, $\rho(g,n)=r(g)$, 
and $\sigma\colon Z\to N$, $\sigma(g,n)=n$. 
We also define a left action of the groupoid $\Gc\tto\Gc^{(0)}$ on $Z$ by 
$$(h,(g,n))\mapsto h\cdot (g,n):=(hg,n)\text{ if }d(h)=\rho(g,n),$$ 
and a right action of the groupoid $\theta^{\pullback}(\Gc)\tto N$ on $Z$ 
by 
$$((g,n),(n_2,h,n_1))\mapsto(g,n)\cdot (n_2,h,n_1):=(gh,n_1)\text{ if }\sigma(g,n)=r(n_2,h,n_1).$$ 
Then $Z$ is a $(\Gc,\theta^{\pullback}(\Gc))$-equivalence. 
\end{theorem}

The proof is based on several lemmas. 

\begin{lemma}\label{welld}
The left and right actions on $Z$ from the statement are well defined. 
\end{lemma}

\begin{proof}
For the left action, the condition $d(h)=\rho(g,n)$ means $d(h)=r(g)$, hence $hg\in\Gc$ is well defined 
and moreover $d(hg)=d(g)=\theta(n)$, hence $h\cdot (g,n)=(hg,n)\in Z$. 
For the right action, the condition $r(n_2,h,n_1)=\sigma(g,n)$ means $n_2=n$, 
hence $r(h)=\theta(n_2)=\theta(n)=d(g)$, and then $gh\in\Gc$ is well defined. 
Moreover $d(gh)=d(h)=\theta(n_1)$, hence indeed $(g,n)\cdot (n_2,h,n_1)=(gh,n_1)\in Z$. 
\end{proof}

\begin{lemma}\label{projs}
Let $\varphi\colon X\to A$ and $\psi\colon Y\to A$ be continuous maps,  
and define $B:=\{(x,y)\in X\times Y\mid \varphi(x)=\psi(y)\}\subseteq X\times Y$. 
If $\psi$ is open and surjective, 
then the Cartesian projection $p_1\colon B\to X$, $(x,y)\mapsto x$, is an open surjective map. 
\end{lemma}

\begin{proof} 
Since $\psi$ is surjective, it is clear that also $p_1$ is surjective. 
Thus $p_1$ is continuous and surjective and then,  
reasoning by contradiction, 
it is straightforward to check that the map $p_1$ is open if and only if it has the limit covering property, 
that is, whenever $\lim\limits_{j\in J}x_j=x$ in~$X$ and $(x,y)\in B$,  
there exist a subnet $\{x_i\}_{i\in I}$ and a net $\{y_i\}_{i\in I}$ with $(x_i,y_i)\in B$ for every $i\in I$ 
and $\lim\limits_{i\in I}y_i=y$ in~$Y$. 

Since $(x,y)\in B$, we have $\varphi(x)=\psi(y)$, 
and thus continuity of $\varphi$ implies 
$\lim\limits_{j\in J}\varphi(x_j)=\psi(y)\in A$. 
Now, using the aforementioned limit covering property for the continuous open surjective map $\psi\colon Y\to A$, 
we obtain a subnet $\{x_i\}_{i\in I}$ of $\{x_j\}_{j\in J}$ and a net $\{y_i\}_{i\in I}$ in $Y$ with $\lim\limits_{i\in I}y_i=y$ and $\psi(y_i)=\varphi(x_i)$ for every $i\in I$. 
Then $(x_i,y_i)\in B$ for every $i\in I$, and thus the limit covering property of the Cartesian projection $p_1$ indeed holds true, which completes the proof. 
\end{proof}

\begin{lemma}\label{opens}
Both $\rho\colon Z\to\Gc^{(0)}$ and $\sigma\colon Z\to N$ are open maps. 
\end{lemma}

\begin{proof}
The map $\sigma\colon Z=\Gc \times_{\Gc^{(0)}} N\to N$ is exactly the Cartesian projection on the second factor. 
Since the map $d\colon \Gc \to{\Gc^{(0)}}$ is open and surjective, it then follows by Lemma~\ref{projs} that $\sigma$ is an open map. 

On the other hand, since the map $\theta\colon N\to \Gc^{(0)}$ is continuous, open and surjective, 
the Cartesian projection on the first factor 
$p_1\colon Z=\Gc \times_{\Gc^{(0)}} N\to \Gc$ is an open map, 
by Lemma~\ref{projs} again. 
As  $r\colon \Gc \to{\Gc^{(0)}}$ is an open map by the definition of a topological groupoid, 
it follows that their composition $r\circ p_1=\rho$ is an open map. 
\end{proof}

We now use terminology from \cite[Sect. 2]{MRW87} again. 

\begin{lemma}\label{leftprinc}
$Z$ is a left principal $\Gc$-space.  
\end{lemma}

\begin{proof}
It follows by Lemma~\ref{opens} that the structural map $\rho\colon Z\to\Gc^{(0)}$ is open, 
so $Z$ is a $\Gc$-space. 
Then we must check that the following conditions are satisfied, 
where $\Gc\times_{\Gc^{(0)}}Z=\{(g,z)\in\Gc\times Z\mid d(g)=\rho(z)\}$: 
\begin{enumerate}
\item The action of $\Gc$ on $Z$ is free, that is, if $(h,z)\in\Gc\times_{\Gc^{(0)}}Z$ 
and $h\cdot z=z$, then $h\in\Gc^{(0)}$. 
\item The action of $\Gc$ on $Z$ is proper, that is, the map $\Phi\colon\Gc\times_{\Gc^{(0)}}Z\to Z\times Z$, 
$(h,z)\mapsto (h\cdot z,z)$ is proper. 
\end{enumerate}
For the first condition, let $z=(g,n)\in Z$ and $h\in\Gc$ with $d(h)=\rho(z)$ and $h.z=z$. 
One has $\rho(z)=r(g)$, hence $d(h)=r(g)$. 
Also $(hg,n)=h\cdot z=z=(g,n)$, hence $hg=g$, and this implies $h=r(g)\in\Gc^{(0)}$. 

For the second of the above conditions, 
we first describe the set $\Gc\times_{\Gc^{(0)}}Z\subseteq\Gc\times Z\subseteq\Gc\times\Gc\times N$. 
If $h,g\in\Gc$ and $n\in N$, then one has 
$(h,(g,n))\in \Gc\times_{\Gc^{(0)}}Z$ if and only if 
$(g,n)\in Z$ and $d(h)=\rho(g,n)$, which is equivalent to $d(g)=\theta(n)$ and $d(h)=r(g)$. 
Hence 
$$\Gc\times_{\Gc^{(0)}}Z=\{(h,g,n)\in\Gc\times\Gc\times N\mid d(g)=\theta(n)\text{ and }d(h)=r(g)\}.$$
Moreover, for every $(h,g,n)\in \Gc\times_{\Gc^{(0)}}Z$, one has 
$$\Phi(h,g,n)=(h.(g,n),(g,n))=((hg,n),(g,n))\in Z\times Z.$$ 
This shows that $\Phi$ is injective and is a homeomorphism onto its image, 
hence in particular $\Phi$ is a proper map. 
\end{proof}

\begin{lemma}\label{rightprinc}
$Z$ is a right principal $\theta^{\pullback}(\Gc)$-space.  
\end{lemma}

\begin{proof}
As in the proof of Lemma~\ref{leftprinc}, 
it follows by Lemma~\ref{opens} that the structural map $\sigma\colon Z\to(\theta^{\pullback}(\Gc))^{(0)}=N$ is open, 
so $Z$ is a $\theta^{\pullback}(\Gc)$-space. 
Now we will check that the following conditions are satisfied, 
where $Z\times_N\theta^{\pullback}(\Gc)=\{(z,\gamma)\in Z\times \theta^{\pullback}(\Gc)\mid \sigma(z)=r(\gamma)\}$: 
\begin{enumerate}
\item The action of $\theta^{\pullback}(\Gc)$ on $Z$ is free, that is, if $(z,\gamma)\in Z\times_N\theta^{\pullback}(\Gc)$ 
and $z\cdot \gamma=z$, then $\gamma\in(\theta^{\pullback}(\Gc))^{(0)}$. 
\item The action of $\theta^{\pullback}(\Gc)$ on $Z$ is proper, that is, 
the map $\Psi\colon Z\times_N\theta^{\pullback}(\Gc)\to Z\times Z$, 
$(z,\gamma)\mapsto (z\cdot \gamma,z)$ is proper. 
\end{enumerate}
For the first condition, let $z=(g,n)\in Z$ and $\gamma=(n_2,h,n_1)\in\theta^{\pullback}(\Gc)$ with $\sigma(z)=r(\gamma)$ and $z\cdot \gamma=z$. 
One has $\sigma(z)=n$ and $r(\gamma)=n_2$, hence $n=n_2$. 
Also $(gh,n_1)=z\cdot \gamma=z=(g,n)$, hence $n=n_1$ and $gh=g$. 
This implies $n_1=n_2=n$ and $h=d(g)\in\Gc^{(0)}$, hence $\gamma=(n_2,h,n_1)\in(\theta^{\pullback}(\Gc))^{(0)}$. 

For the second of the above conditions, 
we first describe the set $Z\times_N\theta^{\pullback}(\Gc)\subseteq Z\times\theta^{\pullback}(\Gc)
\subseteq(\Gc\times N)\times(N\times\Gc\times N)$. 
If $h,g\in\Gc$ and $n,n_1,n_2\in N$, then one has 
$((g,n),(n_2,h,n_1))\in Z\times_N\theta^{\pullback}(\Gc)$ if and only if 
$(g,n)\in Z$, $(n_2,h,n_1))\in\theta^{\pullback}(\Gc)$, 
 and $\sigma(g,n)=r(n_2,h,n_1)$.  
This is equivalent to $d(g)=\theta(n)$, $d(h)=\theta(n_1)$, $r(h)=\theta(n_2)$, and $n=n_2$. 
Hence 
$$\begin{aligned}
Z\times_N\theta^{\pullback}(\Gc)=\{( & g,n,n,h,n_1)\in\Gc\times N\times N\times\Gc\times N\mid \\
& d(g)=r(h)=\theta(n)\text{ and } d(h)=\theta(n_1)\}.
\end{aligned}$$
Moreover, for every $(g,n,n,h,n_1)\in Z\times_N\theta^{\pullback}(\Gc)$, one has 
$$\Psi(g,n,n,h,n_1)=((g,n)\cdot (n,h,n_1),(g,n))=((gh,n_1),(g,n))\in Z\times Z.$$ 
This shows that $\Psi$ is injective and is a homeomorphism onto its image, 
hence $\Psi$ is a proper map. 
\end{proof}

\begin{lemma}\label{rightleaves}
The map 
$\tilde\rho\colon Z/\theta^{\pullback}(\Gc)\to \Gc^{(0)}$,
 $z\cdot \theta^{\pullback}(\Gc)\mapsto \rho(z)$ is well defined and is bijective.   
\end{lemma}

\begin{proof}
If $z=(g,n)$ and $\gamma=(n_2,h,n_1)$ with $(z,\gamma)\in Z\times_N\theta^{\pullback}(\Gc)$, then 
$n_2=n$, $d(g)=r(h)=\theta(n)$, $d(h)=\theta(n_1)$, and $z\cdot\gamma=(gh,n_1)$ 
(see the proof of Lemma~\ref{rightprinc}), 
hence 
$\rho(z\cdot\gamma)=\rho(gh,n_1)=r(g)=\rho(z)$.  
This shows that the map $\tilde\rho$ is well defined. 
Moreover, $\tilde\rho$ is surjective because $\rho$ is clearly surjective. 

To check that $\tilde\rho$ is injective, let $z=(g,n)$ and $z'=(g',n')$ in $Z$ with $\rho(z)=\rho(z')$, 
that is, $r(g)=r(g')$. 
We have to find $\gamma=(n_2,h,n_1)\in \theta^{\pullback}(\Gc)$ with $(z,\gamma)\in Z\times_N\theta^{\pullback}(\Gc)$ 
and $z'=z\cdot \gamma$. 
As above, these two conditions are equivalent to 
$n_2=n$, $d(g)=r(h)=\theta(n)$, $d(h)=\theta(n_1)$, $g'=gh$, and $n'=n_1$, 
hence we obtain the unique solution $\gamma=(n,g^{-1}g',n')\in N\times\Gc\times N$, which is well defined since $r(g)=r(g')$. 
We note that $d(g^{-1}g')=d(g')=\theta(n')$ since $z'\in Z$, and $r(g^{-1}g')=r(g^{-1})=d(g)=n$ since $z\in Z$, 
and this shows that indeed $h\in \theta^{\pullback}(\Gc)$, as we wished for.
\end{proof}

\begin{lemma}\label{leftleaves}
The map 
$\tilde\sigma\colon \Gc\setminus Z\to N$, $\Gc\cdot z\mapsto \sigma(z)$ is well defined and is bijective.   
\end{lemma}

\begin{proof}
If $h\in\Gc$ and $z=(g,n)\in Z$ with $(h,z)\in \Gc\times_{\Gc^{(0)}}Z$, then 
$d(h)=r(g)$, $d(g)=\theta(n)$, and $h\cdot z=(hg,n)$, 
hence 
$\sigma(h\cdot z)=n=\sigma(z)$.  
This shows that the map $\tilde\sigma$ is well defined. 
Moreover, $\tilde\sigma$ is surjective because $\sigma$ is clearly surjective. 

To check that $\tilde\sigma$ is injective, let $z=(g,n)$ and $z'=(g',n')$ in $Z$ with $\sigma(z)=\sigma(z')$, 
that is, $n=n'$. 
We must find $h\in\Gc$ with $(h,z)\in\Gc\times_{\Gc^{(0)}} Z$ and $h\cdot z=z'$, 
which is equivalent to $d(h)=r(g)$ and $(hg,n)=(g',n')$. 
Using that $z,z'\in Z$, we obtain $d(g)=\theta(n)=\theta(n')=d(g')$, 
hence one has the unique solution $h=g'g^{-1}\in\Gc$, which indeed satisfies $d(h)=r(g)$, and we are done. 
\end{proof}

\begin{proof}[Proof of Theorem~\ref{Morita1}]
To prove that $Z$ is a $(\Gc,\theta^{\pullback}(\Gc))$-equivalence, 
we must check the conditions of \cite[Def. 2.1]{MRW87}, that is: 
\begin{enumerate}[(i)]
\item $Z$ is a left principal $\theta^{\pullback}(\Gc)$-space.  
\item $Z$ is a right principal $\theta^{\pullback}(\Gc)$-space.  
\item The actions of $\Gc$ and $\theta^{\pullback}(\Gc)$ on $Z$ commute.  
\item The map $\rho$ induces a bijection of $Z/\theta^{\pullback}(\Gc)$ onto $\Gc^{(0)}$. 
\item The map $\sigma$ induces a bijection of $\Gc\setminus Z$ onto $(\theta^{\pullback}(\Gc))^{(0)}$. 
\end{enumerate}
The first two of these conditions follow by Lemmas \ref{leftprinc} and \ref{rightprinc}, respectively. 
An inspection of the actions of $\Gc$ and $\theta^{\pullback}(\Gc)$ on $Z$ 
shows that the third condition follows by associativity of the multiplication in the groupoid~$\Gc$. 
Finally, the fourth and fifth of the above conditions follow by Lemmas \ref{leftleaves} and \ref{rightleaves}, 
and this completes the proof. 
\end{proof}

\begin{remark}
\normalfont
In Theorem~\ref{Morita1}, the $(\Gc,\theta^{\pullback}(\Gc))$-equivalence $Z$ is actually 
an essential equivalence in the sense of \cite{Mr96} (see particularly \cite[Cor. II.1.7]{Mr96}). 
This can be seen from the proofs of Lemmas \ref{leftprinc} and \ref{rightprinc}, 
where each of the maps $\Phi$ and $\Psi$ is a homeomorphisms onto its image, 
rather than only a proper map. 
\end{remark}

\begin{corollary}\label{Morita2}
Let $\Gc\tto \Gc^{(0)}$ be a locally compact groupoid having a left Haar system.  
If $N$ is a second countable, locally compact topological space, 
and  
$\theta\colon N\to \Gc^{(0)}$ is a continuous open surjective 
map, 
then $\theta^{\pullback}(\Gc)\tto N$ is a locally compact groupoid with a left Haar system, 
and the $C^*$-algebras $C^*(\Gc)$ and $C^*(\theta^{\pullback}(\Gc))$ are Morita equivalent. 
\end{corollary}

\begin{proof}
It follows by Theorem~\ref{Morita1} that there exists a $(\Gc,\theta^{\pullback}(\Gc))$-equivalence. 
Now, since $\Gc$ has a Haar system, it follows 
by \cite[Th. 2.1]{Wi15} that also $\theta^{\pullback}(\Gc)$ has a Haar system, 
and then \cite[Th. 2.8]{MRW87} implies that the $C^*$-algebras $C^*(\Gc)$ and $C^*(\theta^{\pullback}(\Gc))$ are Morita equivalent. 
\end{proof}

One of the conditions imposed in Theorem~\ref{Morita1} is that certain maps be open and surjective. 
A convenient tool for checking these conditions 
is the notion of local cross-section:   
A map between topological spaces $\theta\colon X\to Y$ has \emph{local cross-sections} 
if for every $y\in Y$ and every $x\in\theta^{-1}(y)$ there exist an open neighborhood $V$ of~$y$ 
and a continuous map $\tau\colon V\to X$ with $\tau(y)=x$ and $\theta\circ\tau=\id_V$. 
The next lemma shows that these maps are in particular open and surjective, and  records some other basic properties, for later use.

\begin{lemma}\label{local}
	If $\theta\colon X\to Y$ is a map between topological spaces, then one has: 
	\begin{enumerate}[(i)]
		\item\label{local_item1} 
		If $X$ and $Y$ are smooth manifolds and $\theta$ is a surjective submersion, then $\theta$ has local cross sections. 
		\item\label{local_item2} 
		If $\theta$ has local cross-sections, then it is surjective and is an open map. 
		\item\label{local_item3} 
		If $\theta$ has local cross-sections and $Y_0\subseteq Y$, 
		then the map $\theta\vert_{\theta^{-1}(Y_0)}\colon \theta^{-1}(Y_0)\to Y_0$ has local cross-sections. 
		\item\label{local_item4} 
		If $\theta$ has local cross-sections, $\eta\colon Z\to Y$ is any continuous map, and we define 
		the topological space $\eta^*(X):=\bigcup\limits_{z\in Z}\{z\}\times\theta^{-1}(\eta(z))\subseteq Z\times X$, 
		then the map $\pi\colon \eta^*(X)\to Z$, $(z,x)\mapsto z$, has local cross-sections. 
	\end{enumerate}
\end{lemma}

\begin{proof}
	Assertions \eqref{local_item1}--\eqref{local_item3} are clear. 
	For Assertion~\eqref{local_item4} let $z_0\in Z$ and $(z_0,x_0)\in\pi^{-1}(z_0)=\{z_0\}\times\theta^{-1}(\eta(z_0))$ arbitrary. 
	Then $\theta(x_0)=\eta(z_0)=:y_0\in Y$, hence by the hypothesis on $\theta$, 
	there exist an open neighborhood $V$ of~$y_0$ 
	and a continuous map $\tau\colon V\to X$ with $\tau(y_0)=x_0$ and $\theta\circ\tau=\id_V$. 
	The set $U:=\eta^{-1}(V)$ is an open neighborhood of $z_0$.  
	Moreover, as $\eta(U)\subseteq V$, 
	the map $\sigma\colon U\to \eta^*(X)$, $\sigma(z):=(z,\tau(\eta(z)))$ is well defined, continuous, 
	and satisfies $\pi\circ\sigma=\id_U$ and $\sigma(z_0)=(z_0,\tau(\eta(z_0)))=(z_0,x_0)$. 
	This completes the proof. 
\end{proof}

Results as Corollary~\ref{Morita3} below were already indicated in \cite[Ex.~2.7]{MRW87} and \cite[Th.~2.5]{MW95}.

\begin{corollary}\label{Morita3}
Let $\Gc\tto \Gc^{(0)}$ be a locally compact groupoid 
with its 
quotient map $q\colon\Gc^{(0)}\to\Gc\setminus\Gc^{(0)}$ 
onto its set of orbits.  
Assume that one has a locally closed subset $\Xi\subseteq \Gc^{(0)}$
for which the map $q\vert_\Xi\colon\Xi\to \Gc\setminus\Gc^{(0)}$ is a homeomorphism, and denote $\theta= (q\vert_\Xi)^{-1}\circ q\colon \Gc^{(0)} \to 
\Xi$. 
Also assume
that one has a continuous map 
$\sigma\colon \Gc^{(0)}\to \Gc$ 
with $d\circ\sigma=\id$ and $r\circ \sigma =\theta$. 
Then define the bundle of isotropy groups $\Gamma:=\bigsqcup\limits_{x\in\Xi}\Gc(x)$, 
and let $\Pi\colon \Gamma\to\Xi$ be the canonical projection with $\Pi^{-1}(x)=\Gc(x)$ for all $x\in\Xi$. 
If $\Pi$ has local cross-sections, 
then one has: 
\begin{enumerate}[(i)]
\item $\Pi\colon\Gamma\to\Xi$ is a group bundle and the groupoid $\Gc\tto\Gc^{(0)}$ has a left Haar system.  
\item The set $\Xi$ is closed and the  map $\theta:=(q\vert_\Xi)^{-1}\circ q\colon\Gc^{(0)}\to\Xi$ is continuous, open, and surjective. 
\item The locally compact groupoids $\Gc\tto\Gc^{(0)}$ and $\theta^{\pullback}(\Gamma)\tto\Gc^{(0)}$ are isomorphic. 
\item The $C^*$-algebras $C^*(\Gc)$ and $C^*(\Gamma)$ are Morita equivalent. 
\end{enumerate}
\end{corollary}

\begin{proof}
The set $\Xi$ is locally compact with its induced topology and intersects every $\Gc$-orbit at exactly one point. Since $q\vert_\Xi$ is a homeomorphism, it follows that $\Gc\setminus\Gc^{(0)}$ is Hausdorff and then it is straightforward to check that $\Xi$ is closed in $\Gc^{(0)}$, as noted for instance in \cite[p.~113]{MW95}. 

The pullback of $\Gamma$ by $\theta$
$$\theta^{\pullback}(\Gamma)=\{(x,h,y)\in \Gc^{(0)}\times \Gamma\times \Gc^{(0)}\mid \theta(x)=\Pi(h)=\theta(y)\}.$$
It is straightforward to check that 
the map 
\begin{equation*}
\Phi\colon\Gc\to \theta^{\pullback}(\Pi),\quad 
\Phi(g):=(r(g),\sigma(r(g))g\sigma(d(g))^{-1},d(g))
\end{equation*}
is an isomorphism of topological groupoids, with its inverse  
\begin{equation*}
\Phi^{-1}\colon \theta^{\pullback}(\Pi)\to \Gc,\quad 
\Phi^{-1}(x,h,y)=\sigma(y)^{-1}h\sigma(x)
\end{equation*}
On the other hand, let us note that both maps $\Pi$ and $\theta$ are continuous, open, and surjective: 
These properties of $\Pi$ follow by Lemma~\ref{local}\eqref{local_item2}, while the same assertion for $\theta$ is a consequence of the well-known fact that $q$ is open and surjective. 
(Recall that for any open set $U\subseteq\Gc^{(0)}$ one has $q^{-1}(q(U))=r(d^{-1}(U))$, which is an open set in $\Gc^{(0)}$, 
since $r$ is an open map and $d$ is continuous. 
Then use the definition of the quotient topology on $\Gc\setminus\Gc^{(0)}$.) 

Then, using Corollary~\ref{Morita2} for the groupoid $\Gc=\Gamma\tto\Xi$ that has a Haar system 
by Remark~\ref{Re91_Lemma1.3}, 
we can see that the groupoid $\theta^{\pullback}(\Gamma)\tto \Gc^{(0)}$ in turn has a Haar system 
and its $C^*$-algebra is Morita equivalent to $C^*(\Gamma)$. 
Therefore, as we have already seen that $\Gc$ is isomorphic to $\theta^{\pullback}(\Gamma)$, 
it follows that $\Gc$ has a Haar system 
and its $C^*$-algebra is Morita equivalent to $C^*(\Gamma)$, which concludes the proof. 
\end{proof}

\begin{remark}
\normalfont
The method of proof of Corollary~\ref{Morita3} goes back to \cite{MRW87} and \cite{Bu03} in the special case of transitive groupoids. 
\end{remark}

For the following corollary we recall that for any locally compact group $G$ we denoted by 
$\Kc(G)$ the compact space of all closed subgoups of $G$ with its Fell topology. (see Section~\ref{Sect3}). 
For results of this type on noncommutative $C^*$-dynamical systems we refer to 
\cite[Prop.~8.7]{Wi07}.

\begin{corollary}\label{Morita4}
Let $\alpha\colon G\times X\to X$, $(g,x)\mapsto g\cdot x$, 
be a continuous action of a locally compact group on a locally compact space, 
quotient map $q\colon X\to G \setminus X$ 
onto its set of orbits. 
We assume the following: 
\begin{itemize}
\item One has a locally closed subset $\Xi\subseteq X$ 
for which the map $q\vert_\Xi\colon\Xi\to G\setminus X$ is a homeomorphism. 
\item The map $\Xi\to \Kc(G)$, $x\mapsto G(x):=\{g\in G\mid g\cdot x=x\}$, is continuous. 
\item One has a continuous map $\sigma_0\colon X\to G$ with $\sigma_0(x)\cdot x=\theta(x)$ for all $x\in X$, where $\theta:=(q\vert_\Xi)^{-1}\circ q\colon X\to\Xi$.
\end{itemize}
Let $\Gamma:=\bigcup\limits_{x\in\Xi}\{x\}\times G(x)$, 
and $\Pi\colon \Gamma\to\Xi$ with $\Pi(\{x\}\times G(x))=\{x\}$ for all $x\in\Xi$. 

Then $\Pi\colon\Gamma\to\Xi$ is a group bundle. 
If moreover  $\Pi$ has local cross sections, 
then one has: 
\begin{enumerate}[(i)]
\item\label{Morita4_item1} 
The set $\Xi$ is closed, and the maps $\theta$ and $q$ are continuous and have local cross-sections. 
\item\label{Morita4_item2} 
The locally compact groupoids $G\times X\tto X$ and $\theta^{\pullback}(\Gamma)\tto X$ are isomorphic. 
\item\label{Morita4_item3} 
The $C^*$-algebras $G\ltimes\Cc_0(X)$ and $C^*(\Gamma)$ are Morita equivalent.  
\item\label{Morita4_item4} If the locally compact group $G$ is amenable, then $C^*(\Gamma)$ is a $\Cc_0(\Xi)$-algebra 
that is $\Cc_0(\Xi)$-linearly $*$-isomorphic to 
the $C^*$-algebra of global cross sections of a continuous $C^*$-bundle whose fibers are $\{C^*(G(x))\}_{x\in\Xi}$. 
\end{enumerate}
\end{corollary}

\begin{proof}
We check that the conditions of Corollary~\ref{Morita3} are satisfied by the groupoid 
$\Ac:=G\times X\tto X=:\Ac^{(0)}$, with $d(g,x)=x$ and $r(g,x)=g\cdot x$, 
defined by the action of $G$ on~$X$. 
Any Haar measure on the locally compact group $G$ defines a left Haar system on $\Ac$, 
and it is well known that $C^*(\Ac)$ is $*$-isomorphic to the crossed product $G\ltimes\Cc_0(X)$. 

The map $\sigma\colon X\to G\times X$, $\sigma(x):=(\sigma_0(x),x)$, is continuous, and satisfies $d\circ\sigma=\id$ and  $r\circ\sigma=\theta$.

To see that $q$ has local cross-sections, define $\gamma:=(q\vert_\Xi)^{-1}\colon G\setminus X\to\Xi$, 
and select $x_0\in X$ arbitrary. 
Then $\xi_0:=\gamma(G\cdot x_0)\in X$ satisfies $G\cdot x_0\cap\Xi=\{\xi_0\}$, 
hence one can find $g\in G$ with $g\cdot \xi_0=x_0$. 
Defining $\gamma_g\colon G\setminus X\to X$, $\gamma_g(G\cdot x):=g.\gamma(G\cdot x)$, 
one obtains a continuous map satisfying $q\circ\gamma_g=\id$ and
 $\gamma_g(G\cdot x_0)=g\cdot \xi_0=x_0$. 

Now define $\theta_0\colon\Xi\to\Sc$, $\theta_0(\xi):=G(\xi)$. 
This map is continuous by hypothesis and one has $\Gamma=\theta_0^*(\Tc)$, 
where $p\colon \Tc\to\Sc$ is the group bundle from Lemma~\ref{bund1}. 
Then $\Pi\colon\Gamma\to\Xi$ is a group bundle by Lemma~\ref{bund2}, 
and this completes the proof of Assertions \eqref{Morita4_item1}--\eqref{Morita4_item3}. 

For Assertion~\eqref{Morita4_item4} we still need to check that if the group $G$ is amenable, 
then $\Pi\colon\Gamma\to\Xi$ is a strong group bundle. 
To this end we recall from Proposition~\ref{strong} that $p\colon\Tc(G)\to\Kc(G)$ is a strong group bundle, 
hence $q\colon\eta^*(\Tc(G))\to\Xi$ is in turn a strong group bundle by Lemma~\ref{bund2bis}\eqref{bund2bis_item2}, 
where $\eta\colon\Xi\to \Kc(G)$, $\eta(x):=G(x)$. 
Here one has 
$\eta^*(\Tc(G))=\{(x,(K,g))\in\Xi\times\Tc(G)\mid \theta(x)=K\}$ 
hence there is the topological isomorphism of group bundles over~$\Xi$, 
$$\eta^*(\Tc(G))\to\Gamma,\quad (x,(K,g))\mapsto (x,g)$$
with its inverse $\Gamma\to\eta^*(\Tc(G))$, $(x,g)\mapsto(x,(\theta(x),g))$. 
Therefore $\Pi\colon\Gamma\to\Xi$ is a strong group bundle, 
and we are done. 
\end{proof}

We point out here that  in the proof of Corollary~\ref{th_exp} we show that  
the hypotheses of the above Corollary~\ref{Morita4} are satisfied for sufficiently small pieces of the coadjoint dynamical system of any exponential solvable Lie group.

\begin{proposition}\label{29Nov2015}
Let $\Tc\to S$ be any group bundle having local cross-sections 
and $\theta\colon N\to S$ be any continuous surjective 
map with local cross-sections. 
Then $\theta^{\pullback}(\Tc)\tto N$ is a regular groupoid. 
\end{proposition}

\begin{proof} 
By Remark~\ref{Re91_Lemma1.3}, the group bundle $\Tc\to S$ has a Haar system. 
It then follows by  
Corollary~\ref{Morita2}
that $\theta^{\pullback}(\Tc)\tto N$ is a locally compact groupoid with a Haar system. 
Therefore, according to Definition~\ref{regular}, it remains to check that 
the isotropy subgroupoid of $\theta^{\pullback}(\Tc)\tto N$ is a group bundle. 
But it is easily seen that the isotropy subgroupoid of $\theta^{\pullback}(\Tc)\tto N$ is $\theta^*(\Tc)\tto N$, 
which is a group bundle by Lemma~\ref{bund2}. 
This completes the proof. 
\end{proof}

\subsection*{Piecewise pullback}

We now introduce a special class of groupoids which prove to be suitable for describing 
the coadjoint transformation groups associated to exponential solvable Lie groups. 
The importance of these groupoids stems from Corollary~\ref{Morita2}, Proposition~\ref{29Nov2015},  
and Theorem~\ref{Cstar_grpd}. 
Loosely speaking, these results together imply that the dual spaces of $C^*$-algebras of these groupoids can be computed 
from $C^*$-algebras of group bundles.

\begin{definition}\label{ppb_def}
\normalfont
Let $\Gc \tto \Gc^{(0)}$ be any locally compact groupoid. 
We say that $\Gc$ is a {\em piecewise pullback of group bundles} 
with \emph{pieces} $V_k$ for $k=1,\dots,n$
if the following conditions are satisfied.
\begin{enumerate}[(i)]
 \item There exists an increasing family
$$\emptyset = U_0\subseteq U_1\subseteq\cdots \subseteq U_n=\Gc^{(0)}$$
 where $U_k$ are open $\Gc$-invariant subsets of $\Gc^{(0)}$ and 
$$V_k=U_k\setminus U_{k-1}$$
 for $k=1,\dots,n$. 
 \item For $k=1,\dots,n$,  
  there exist an open continuous surjective map $\theta_k \colon V_k \to S_k$ having local cross-sections and  
 a group bundle $\Tc_k \to S_k$ for which one has an isomorphism of topological groupoids 
 $\Gc_{V_k}\simeq \theta_k^{\pullback}(\Tc_k)$. 
\end{enumerate}
\end{definition}

\begin{remark}\label{orbits}
\normalfont
In Definition \ref{ppb_def}, 
the orbits of $\Gc$ are exactly 
the sets $\theta_k\sp{-1}(x)$ for $x \in S_k$ and $k=1,\dots,n$. 
This shows that for every $k$, the orbits of the reduced groupoid $\Gc_{V_k}$ 
are closed, and moreover the orbit space of that groupoid 
is homeomorphic to $S_k$, by 
Proposition~\ref{beta}. 

We also note that $\Gc_{V_1}$ is a pair groupoid if and only if the set $\Tc_1$ is a singleton. 
\end{remark}

\begin{theorem}\label{ppb_th}
Let $\Gc\tto \Gc^{(0)}$ be a piecewise pullback of group bundles that has a left Haar system. 
Then:
\begin{enumerate}[(i)]
\item\label{ppb_th_item1} The orbits of $\Gc$ are locally closed subsets of $\Gc^{(0)}$. 
\item\label{ppb_th_item3} The $C^*$-algebra of $\Gc$ has a sequence of closed two-sided ideals 
$$\{0\}=\Jc_0\subseteq \Jc_1\subseteq\cdots\subseteq \Jc_n=C^*(\Gc)$$
such that for $k=1,\dots,n$ the subquotient $\Jc_k/\Jc_{k-1}$ is Morita equivalent to the $C^*$-algebra 
of sections of an upper semicontinuous $C^*$-bundle whose fibers are $C^*$-algebras of isotropy groups of~$\Gc$, 
and every isotropy group of~$\Gc$ occurs in this way for exactly one value of~$k$.  
\end{enumerate}
If in addition all the isotropy groups of $\Gc$ are amenable, then: 
\begin{enumerate}[(i)]\setcounter{enumi}{2}
\item\label{ppb_th_item2} $\Gc$ is metrically amenable. 
\item\label{ppb_th_item5} If all the isotropy groups of $\Gc$ are of type I, the multiplier algebra of the $C^*$-algebra $\Jc_k/\Jc_{k-1}$ is a $C(\beta(V_k/\Gc))$-algebra, where $V_k$ for are the pieces of $\Gc\tto \Gc^{(0)}$, for $k=1,\dots,n$.
\item\label{ppb_th_item4} The $C^*$-algebra of $\Gc$ is of type~I if and only if all the isotropy groups of~$\Gc$ are of type~I.  
\end{enumerate} 
\end{theorem}

\begin{proof}
Assertion~\eqref{ppb_th_item1} follows directly by Remark~\ref{orbits}. 

For Assertion~\eqref{ppb_th_item3}, use the bijective correspondence from Proposition~\ref{Renault_page101}\eqref{Renault_page101_item1} to define $\Jc_k:=I(U_k)$ for $k=0,\dots,n$. 
Then one has the isomorphism of $C^*$-algebras $\Jc_k/\Jc_{k-1}\simeq C^*(\Gc_{V_k}^{V_k})$. 
On the other hand, one has an isomorphism of topological groupoids 
 $\Gc_{V_k}\simeq \theta_k^{\pullback}(\Tc_k)$ by Definition~\ref{ppb_def}, 
 where $\theta_k\colon V_k\to S_k$ is an open continuous injection and $\Tc_k\tto S_k$ is a group bundle. 
Hence Corollary~\ref{Morita2} 
implies that the $C^*$-algebras $C^*(\Gc_{V_k})$ and $C^*(\Tc_k)$ are Morita equivalent.

Since the orbits of $\Gc$ are locally closed and $\Gc$ is locally compact and second countable, 
it follows by \cite[Th. 2.1((5)$\Leftrightarrow$(4))]{Ra90} that the orbit space of $\Gc$ is a topological space of type~$T_0$. 
Then \cite[Th. 4]{SW13} implies that if the isotropy groups of $\Gc$ are amenable, 
then $\Gc$ is a measurewise amenable groupoid, hence it is also metrically amenable 
(see for instance \cite[Sect. II.3]{Re80}), and this concludes the proof of  Assertion~\eqref{ppb_th_item2}. 

For Assertion~\eqref{ppb_th_item5}, 
it follows by Corollary~\ref{multipl} that $M(C^*(\Gc_{V_k}))$ is a $C(\beta(V_k/\Gc))$-algebra, 
because the orbit space $V_k/\Gc$ is Hausdorff by Remark~\ref{orbits}. 

Finally, since $\Tc_k\tto S_k$ is a a group bundle, it follows by Lemma~\ref{1Cstar_isotrop} that $C^*(\Tc_k)$ is a $\Cc_0(S_k)$-algebra. 
This, along with \cite[Th.~7.1]{Cl07}, also implies Assertion~\eqref{ppb_th_item4}, because if two $C^*$-algebras are Morita equivalent, 
then one of them is of type~I if and only if the other one is. 
\end{proof}

The above Theorem~\ref{ppb_th} provides the motivation for the notion of piecewise pullback of group bundles 
introduced in Definition~\ref{ppb_def}. 
We give some examples of such structures in Theorem~\ref{stratif} and Corollary~\ref{th_exp}. 
Composition series as in Theorem~\ref{ppb_th}\eqref{ppb_th_item3} also play an important role 
in the description of $C^*$-algebras of nilpotent/exponential Lie groups (see \cite{BBL14} and \cite{BB16a}). 
We also mention that infinite composition series of this type exist for all $C^*$-dynamical systems $(A,G,\alpha)$ 
for which the orbit space of the action of $G$ on ${\rm Prim}\,A$ is almost Hausdorff, 
as shown in \cite[Th. 8.16]{Wi07}.

\section{Application to  
linear dynamical systems}\label{Sect5}

Let $G$ be any 
Lie group with a continuous representation $\tau\colon G\to\End(\Vc)$ 
on a finite-dimensional real vector space. 
We say that the action of $G$ on $\Vc$ is a \emph{linear dynamical system}, 
and this transformation group defines as usual a groupoid $G\times\Vc\tto\Vc$, 
whose $C^*$-algebra is the crossed product $G\ltimes\Cc_0(\Vc)$. 
Our main aim here is to investigate some aspects of the way the $C^*$-algebra of that groupoid  
encodes topological properties of  
 the $G$-orbits in $\Vc$ or in $\Vc^*$, or of the orbit space~$G\setminus\Vc:=\{\tau(G)v\mid v\in\Vc\}$.  

Regarding the dual space $\Vc^*$ as an abelian Lie group $(\Vc^*,+)$, 
the Fourier transform defines a $*$-isomorphism $\Cc_0(\Vc)\simeq C^*(\Vc^*,+)$, 
hence one has 
$$G\ltimes\Cc_0(\Vc)\simeq G\ltimes C^*(\Vc^*,+)\simeq C^*(G\ltimes \Vc^*)$$ 
by \cite[Prop. 3.11]{Wi07}, 
where the semidirect product $G\ltimes \Vc^*$ is defined via the contragredient representation 
$\tau^*\colon G\to\End(\Vc^*)$, $\tau^*(g):=\tau(g^{-1})^*$, regarded as a representation of $G$ 
by automorphisms of the abelian group $(\Vc^*,+)$.

The relevance of the above constructions for representation theory of $G$ 
comes from the fact that if the Lie group $G$ is amenable, 
then one has the short exact sequence of amenable Lie groups 
\begin{equation}\label{exact1}
0\to \Vc^*\to G\ltimes \Vc^*\to G\to\1
\end{equation}
which leads to a short exact sequence of $C^*$-algebras 
\begin{equation}\label{exact2}
0\to\Jc\to C^*(G\ltimes \Vc^*)\to C^*(G)\to 0.
\end{equation}
Classical examples of this type are the $ax+b$-group (see Example~\ref{ax+b}) 
or the motion groups (in which $G$ is a compact group). 

In the following we use the notion of module of exponential type in the sense of \cite[Def. 5.3.1]{FL15}. 
For any continuous representation $\tau\colon G\to\End(\Vc)$ of a finite-dimensional Lie group on a finite-dimensional real vector space, we denote by $\Vc_{\CC}:=\CC\otimes_{\RR}\Vc$ the complexification of $\Vc$ and by 
$\gg$ the Lie algebra of $G$, with its complexification $\gg_{\CC}:=\CC\otimes_{\RR}\gg$. 
The differential of the representation $\tau$ extends to a Lie algebra representation $\de\tau\colon\gg\to\End(\Vc_{\CC})$, where $\End(\Vc_{\CC})$
stands for the space of $\CC$-linear maps on $\Vc$.
A weight is an $\RR$-linear functional $\lambda\colon\gg\to\CC$ for which there exists $v\in\Vc_{\CC}$ with $v\ne0$ and $\de\tau(X)v=\lambda(X)v$ for every $X\in\gg$. 
One says that the action of $G$ on $\Vc$ defines 
a \emph{module of exponential type} if for every weight $\lambda$ there exist $c\in\RR\setminus\{0\}$ and $\gamma\in\gg^*$ with $\lambda(X)=(1+\ie c)\gamma(X)$ for every $X\in\gg$. 
If $\Vc=\gg$ and $\tau=\Ad_G$ is the adjoint representation of $G$, 
then the weights of the $G$-module $\gg$ are called roots, and the $G$-module $\gg$ is of exponential type if and only if $G$ is an exponential Lie group (see \cite[Subsect. 5.2.2]{FL15}).

\begin{proposition}\label{expmod}
If $G$ is an exponential solvable Lie group and the representation $\tau\colon G\to\End(\Vc)$ 
defines the structure of a module of exponential type on the finite-dimensional real vector space~$\Vc$, 
then the following assertions hold: 
\begin{enumerate}[(i)]
\item\label{expmod_item1} 
The contragredient representation $\tau^*\colon G\to\End(\Vc^*)$ defines 
the structure of an exponential module on the finite-dimensional real vector space~$\Vc^*$. 
\item\label{expmod_item2} 
The semidirect products $G\ltimes \Vc$ and $G\ltimes\Vc^*$ are exponential solvable Lie groups.  
\item\label{expmod_item3} 
The ideal $\Jc$ from~\eqref{exact2} corresponds to the subset of the dual of $G\ltimes \Vc^*$ 
defined by the coadjoint orbits of functionals that do not vanish identically 
on the abelian ideal~$\Vc^*\subseteq\gg\ltimes\Vc^*$. 
\end{enumerate}
\end{proposition}

\begin{proof}
Denote $n:=\dim\Vc$ and $m:=\dim\gg$, and use Sophus Lie's theorem to 
select a sequence of $G$-invariant subspaces of the complexified vector space $\Vc_{\CC}$, 
$$\{0\}=\Vc_0\subset\Vc_1\subset\cdots\subset\Vc_m=\gg_{\CC}$$ 
with $\dim_{\CC}\Vc_j=j$ for $j=0,\dots,n$.  
Denoting by $\Vc_{\CC}^*$ the space of $\CC$-linear functionals on~$\Vc_{\CC}$ and 
$\Vc_j^\perp:=\{\xi\in\Vc_{\CC}^*\mid\Vc_j\subseteq\Ker\xi\}$, it then follows that 
\begin{equation}\label{expmod_proof_eq1}
\{0\}=\Vc_n^\perp\subset\Vc_{n-1}^\perp\subset\cdots\subset\Vc_1^\perp\subset\Vc_{\CC}^* 
\end{equation}
is a sequence of $G$-invariant subspaces, 
every successive quotient in the above sequence having its dimension over~$\CC$ equal to~1. 

If $\lambda_1,\dots,\lambda_n\colon\gg\to\CC$ are the weights of the $G$-module $\Vc$, hence 
$$(\de\tau(X)-\lambda_j(X))\Vc_j\subseteq\Vc_{j-1}\text{ for }j=1,\dots,n\text{ and all }X\in\gg,$$ 
then it is easily seen from \eqref{expmod_proof_eq1} that the functionals 
$-\lambda_1,\dots,-\lambda_m\colon\gg\to\CC$ are  the weights of the $G$-module $\Vc^*$, 
and this completes the proof of Assertion~\eqref{expmod_item1}, 
by \cite[Def. 5.3.1]{FL15}. 

For Assertion~\eqref{expmod_item2} it suffices to prove that $G\ltimes\Vc$ is an exponential Lie group. 
Denote $m:=\dim\gg$ and use again Sophus Lie's theorem to select a sequence of ideals of 
the complexified Lie algebra $\gg_{\CC}$, 
$$\{0\}=\gg_0\subset\gg_1\subset\cdots\subset\gg_m=\gg_{\CC}$$ 
with $\dim_{\CC}\gg_j=j$ for $j=0,\dots,m$. 
Then 
\begin{equation}\label{expmod_proof_eq2}
\{0\}=\Vc_0\subset\Vc_1\subset\cdots\subset\Vc_n=\Vc_{\CC}
\subset\gg_1\ltimes\Vc_{\CC}\subset\cdots\subset\gg_m\ltimes\Vc_{\CC}=\gg_{\CC}\ltimes \Vc_{\CC}
\end{equation}
is a sequence of $G\ltimes\Vc$-invariant subspaces.  
Let $\mu_1,\dots,\mu_m\colon\gg\to\CC$ be the weights of the $G$-module $\gg$, 
and define $\widetilde{\mu}_1,\dots,\widetilde{\mu}_{m+n}\colon\gg \ltimes \Vc \to\CC$ 
by $\widetilde{\mu}_j(X,v):=\mu_j(X)$ if $j=1,\dots,m$, and $\widetilde{\mu}_j(X,v):=\lambda_{j-m}(X)$ 
if $j=m+1,\dots,m+n$. 
Then, using~\eqref{expmod_proof_eq2}, 
it is easily seen that the functionals $\widetilde{\mu}_1,\dots,\widetilde{\mu}_{m+n}$ 
are the roots of the Lie algebra~$\gg\ltimes\Vc$.

The hypothesis that the Lie group $G$ is exponential is equivalent to the fact that for 
every $j=1,\dots,m$ there exist numbers $\theta_j\in\RR$ and  functionals $\xi_j\in\gg^*$ 
with $[\gg,\gg]\subseteq\Ker\xi_j$ and $\mu_j(X)=(1+\ie\theta_j)\xi_j(X)$ for every $X\in\gg$ 
(see for instance \cite[Subsect. 5.2.2]{FL15}). 
Similarly, the hypothesis that $\Vc$ is an exponential $G$-module 
means that for $k=1,\dots,n$ 
there exist the number $\gamma_k\in\RR$ and the functional $\eta_k\in\gg^*$ 
with $[\gg,\gg]\subseteq\Ker\eta_k$ and $\lambda_k(X)=(1+\ie\gamma_k)\eta_k(X)$ for every $X\in\gg$. 

Then, defining $\widetilde{\theta}_j:=\theta_{j}$  and  $\widetilde{\xi}_j(X,\eta):=\xi_{j}(X)$ if $j=1,\dots,m$, 
and on the other hand $\widetilde{\theta}_j:=\gamma_{j-m}$ and $\widetilde{\xi}_j(X,v):=\eta_{j-m}(X)$
if $j=m+1,\dots,m+n$, for all $X\in\gg_{\CC}$ and $\eta\in\Vc_{\CC}$, 
we obtain $\widetilde{\mu}_j=(1+\ie\widetilde{\theta}_j)\widetilde{\xi}_j$ on $\gg\ltimes\Vc$ for $j=1,\dots,m+n$, 
hence $G\ltimes\Vc$ is an exponential Lie group, by the characterization of exponential Lie groups in terms of roots 
used above. 

For the assertion on the ideal $\Jc$ from~\eqref{exact2}, 
let us consider the exact sequence of Lie algebras 
$$0\to\Vc^*\to\gg\ltimes\Vc^*\mathop{\longrightarrow}\limits^{\pi}\gg\to0.$$
One has $\widehat{G\ltimes\Vc^*}=\widehat{C^*(G\ltimes\Vc^*)}=\widehat{\Jc}\sqcup\widehat{C^*(G)}=\widehat{\Jc}\sqcup\widehat{G}$, 
where $\widehat{G}$ is embedded into $\widehat{G\ltimes\Vc^*}$ as the closed subset 
corresponding (via the method of coadjoint orbits) 
to the set of coadjoint orbits of $G\ltimes\Vc^*$ that are contained in $\pi^*(\gg^*)\subseteq(\gg\ltimes\Vc^*)^*$. 
Since $\Ker\pi=\Vc^*$, it is easily seen that $\pi^*(\gg^*)=\{\xi\in(\gg\ltimes\Vc^*)^*\mid\Vc^*\subseteq\Ker\xi\}$. 
Therefore $\widehat{\Jc}=\widehat{G\ltimes\Vc^*}\setminus\widehat{G}$ corresponds to 
the set of coadjoint orbits of $G\ltimes\Vc*$ that are contained in 
$(\gg\ltimes\Vc^*)^*\setminus\pi^*(\gg^*)=\{\xi\in(\gg\ltimes\Vc^*)^*\mid\Vc^*\not\subset\Ker\xi\}$, 
and this completes the proof. 
\end{proof}

We now begin the preparations for proving Theorem~\ref{stratif}. 

\begin{lemma}\label{triv0}
Let $\Vc$ be a finite-dimensional vector space and $\pi\colon E\to M$ be a continuous locally trivial vector bundle 
over a topological space, 
with its space of continuous global sections denoted by~$\Cc(\pi)$. 
Assume that $\tau\colon\Vc\to\Cc(\pi)$ is a linear map 
and for every $x\in M$ denote $\Vc(x):=\{v\in\Vc\mid (\tau(v))(x)=0\}$. 
If $\dim\Vc(x)$ is locally constant with respect to $x\in M$, then the map $M\to\Gr(\Vc)$, $x\mapsto\Vc(x)$, 
is continuous. 
\end{lemma}

\begin{proof}
The assertion has a local character,
so we may assume that the vector bundle $\pi$ is trivial and
$\dim\Vc(\cdot)$ is constant on $M$.

Now consider the trivial vector bundle $M\times\Vc\to M$ and define
$$\tilde\tau\colon M\times\Vc\to E,\quad \tau(x,v):=(\tau(v))(x).$$
Then it is clear that $\tilde\tau$ is a continuous morphism of vector bundles,
having constant rank,
hence $\Ker\tilde\tau=\bigcup\limits_{x\in M}\{x\}\times\Vc(x)$
is a locally trivial vector bundle
(see for instance \cite[Ch. 3, Th. 8.2]{Hu94}).

Further shrinking $M$, we may assume that $\Ker\tilde\tau\to M$ is
actually a trivial vector bundle.
This implies that there exist a vector space $\Vc_0$ and a continuous map
$\psi\colon M\to\Lc(\Vc_0,\Vc)$
for which $\psi(x)\colon\Vc_0\to\Vc(x)$ is a linear isomorphism for
all $x\in M$.

This directly implies the conclusion. 
\end{proof}

The following result was noted in \cite[Prop. 2.2]{Ri68} 
but we give here a simple alternative proof for the sake of completeness. 

\begin{lemma}\label{triv1}
Let $\alpha\colon G\times M\to M$  
be a smooth action of a Lie group on a manifold. 
Assume that $M_0\subseteq M$ is a subset for which  
the dimension of the isotropy groups $G(x):=\{g\in G\mid \alpha(g,x)=x\}$ is locally constant with respect to $x\in M_0$,  
then the isotropy-algebra map $M_0\to\Gr(\gg)$, $x\mapsto\gg(x)$ is continuous. 
\end{lemma}

\begin{proof}
Let $\Xc(M)$ be the Lie algebra of smooth vector fields on~$M$, 
and denote by $\alpha'\colon\gg\to\Xc(M)$ the morphism of Lie algebras which is 
the infinitesimal action corresponding to~$\alpha$. 
Then for every $x\in M$ define the 
linear map $\tilde\alpha(x)\colon \gg\to T_xM$, $(\tilde\alpha(x))X:=(\alpha'(X))$, 
so that $\Ker(\tilde\alpha(x))=\gg(x)$. 
Thus the assertion follows by Lemma~\ref{triv0} for the continuous vector bundle $\pi\colon E\to M_0$ defined 
as the restriction of the tangent bundle $TM\to M$ to $M_0$. 
\end{proof}

\begin{theorem}\label{stratif}
Let $\Vc$ be a module of exponential type via 
a continuous representation $\tau\colon G\to\End(\Vc)$ 
of an exponential solvable Lie group~$G$. 
Consider the corresponding 
quotient map $q\colon \Vc\to G\setminus\Vc$, $v\mapsto\tau(G)v$, and 
for every $v\in \Vc$ denote $G(v):=\{g\in G\mid\tau(g)v=v\}$. 
Assume that $\Vc=\bigsqcup_{k=1}^n A_k$ is a partition into $G$-invariant subsets 
satisfying the following conditions for $k=1,\dots,n$: 
\begin{enumerate}
\item The set $D_k:=\bigsqcup_{j=1}^k A_j$ is open in $\Vc$. 
\item The map $q\vert_{\Xi_k}\colon\Xi_k\to G\setminus A_k$ is a homeomorphism 
on a suitable locally closed subset $\Xi_k\subseteq A_k$.
\item There  exists a continuous map $\sigma_{0, k} \colon \Xi_k\to G$ with $\sigma_{0,k}(v)\cdot v=\theta_k(v)$ for all $v \in A_k$, where $\theta_k:=(q\vert_\Xi)^{-1}\circ q\colon A_k\to\Xi_k $. 
\item\label{stratif_item3} 
The dimension of $G$-orbits contained in $A_k$ is constant. 
\end{enumerate}
If we define $\Gamma_k:=\bigsqcup\limits_{v\in\Xi_k} \{v\}\times G(v)\subseteq\Xi_k\times G$  
and $\Pi_k\colon \Gamma_k\to\Xi_k$, $\Pi_k(v,g)=v$, 
then 
$\Pi_k\colon \Gamma_k\to\Xi_k$ has local cross-sections for $k=1,\dots,n$, 
and the transformation group $(G,\Vc)$ is a piecewise pullback of group bundles. 
\end{theorem}

\begin{proof}
It follows by the hypothesis \eqref{stratif_item3} along with \cite[Prop. 2.2]{Ri68} 
Lemma~\ref{triv0} 
that the map $A_k\to\Gr(\gg)$, $\xi\mapsto \gg(\xi)$ is continuous. 
Since the exponential map $\exp_G\colon\gg\to G$ is a diffeomorphism, 
it is straightforward to prove that the map $\hg\mapsto\exp_G(\hg)$ 
is a homeomorphism onto its image, from $\Gr_{\rm alg}(\gg)$ onto 
the space of connected closed subgroups of $G$, 
the latter space being regarded as a subspace of the compact space $\Kc(G)$ of all closed subgroups of $G$. 
Here $\Gr_{\rm alg}(\gg)$ is the closed subset of $\Gr(\gg)$ 
consisting of the subalgebras of~$\gg$. 
As for every $v\in V$ its isotropy group $G(v)$ is connected by \cite[Th. 5.3.2]{FL15}, 
we then obtain that 
the map $A_k\to\Kc(G)$, $\xi\mapsto G(\xi)$ is continuous. 

It then follows by Lemmas \ref{bund1} and \ref{bund2} that $\Pi_k\colon \Gamma_k\to\Xi_k$ is a group bundle, 
hence it has a Haar system. 
Moreover, $\Pi_k$ has local cross-sections as a direct consequence of Lemma~\ref{local}, 
of the above topological embedding $\Gr_{\rm alg}(\gg)\hookrightarrow\Kc(G)$  
and of the fact that the tautological bundle over $\Gr(\gg)$ is a locally trivial vector bundle 
hence has local cross-sections. 

Now define $\gamma_k:=(q\vert_{\Xi_k})^{-1}\colon G\setminus A_k\to A_k$, 
which is a continuous map.  
Thus Corollary~\ref{Morita4} applies for the group action $G\times A_k\to A_k$, 
and implies that the locally compact groupoids $G\times A_k\tto A_k$ and $\theta_k^{\pullback}(\Gamma_k)$ are isomorphic, 
where $\theta_k:=\gamma_k\circ q\vert_{A_k}\colon A_k\to\Xi_k$. 
Thus the conditions of Definition~\ref{ppb_def} are satisfied, and this completes the proof. 
\end{proof}

\begin{remark}
\normalfont
In the above proof, in order to compute $\theta_k^{\pullback}(\Pi_k)$ more explicitly, note that for any $\xi,\eta\in A_k$ we have 
$$\begin{aligned}
\theta_k(\zeta)=\theta_k(\eta)
& \iff q(\zeta)=q(\eta) \\
& \iff (\exists\Oc\in G\setminus V)\ \zeta,\eta\in\Oc\subseteq A_k \\
&\iff(\exists \xi\in\Xi_k)\ \zeta,\eta\in\tau(G)\xi, 
\end{aligned}$$
and, if this is the case, then 
$(\zeta,h,\eta)\in \theta_k^{\pullback}(\Pi_k)$ if and only if $\Pi_k(h)=\xi$, 
that is, if and only if $h\in G(\xi)$. 
Therefore 
$$\theta_k^{\pullback}(\Pi_k)=\bigsqcup_{\xi\in\Xi_k}\Oc_\xi\times G(\xi)\times\Oc_\xi$$
where $\Oc_\xi:=\tau(G)\xi\subseteq V$ is the $G$-orbit of $\xi\in\Xi_k$. 
\end{remark}

\begin{corollary}
\label{th_exp}
The coadjoint dynamical system of any exponential solvable Lie group is a piecewise pullback of group bundles.   
There exists a decomposition into pieces such that 
the $C^*$-algebra of the reduction of the coadjoint dynamical system  
to its first piece is commutative,  
while the $C^*$-algebras of the reductions of the dynamical system to these pieces 
are Morita equivalent to continuous fields of sections of continuous $C^*$-bundles whose fibers are $C^*$-algebras of coadjoint isotropy groups. 
\end{corollary}

\begin{proof}
Let $G$ be an exponential solvable Lie group, with its coadjoint dynamical system $(G,\gg^*,\Ad_G^*)$. 
Then $\gg^*$ is a $G$-module of exponential type. 
Moreover, 
it is known from \cite[Th. 2.8]{Cu92} and \cite[Cor.~4.12]{ACD09} that there exists a partition $\gg^*=\bigsqcup_{k=1}^n A_k$ into $G$-invariant, 
subsets 
satisfying the conditions of Theorem~\ref{stratif}. 

Resuming the notation of Theorem~\ref{stratif}, we now show that 
the fibers of the group bundle $\Pi_1\colon \Gamma_1\to\Xi_1$ are 
connected abelian Lie groups, hence its $C^*$-algebra is commutative. 
It follows by \cite{DV69} that for every $\xi\in \Xi_1$ ($\subseteq A_1$) the coadjoint isotropy Lie algebra $\gg(\xi)$ is abelian. 
But the hypothesis that $G$ is an exponential solvable Lie group implies  that the coadjoint isotropy group $G(\xi)$ 
is connected. 
Since its Lie algebra $\gg(\xi)$ is abelian, it follows that $G(\xi)$ is an abelian Lie group $(\RR^b,+)$.  
An application of Corollary~\ref{Morita4} then completes the proof.
\end{proof}

\begin{remark}\label{final2}
\normalfont
The ``first piece'' of the coadjoint dynamical system referred to in 
Corollary~\ref{th_exp} 
 holds a central role in representation theory of the exponential Lie group under consideration, for instance because its corresponding set of coadjoint orbits supports the Plancherel measure.  
In order to provide a more clear picture of that piece
we briefly discuss here how it can be computed for two types of exponential solvable Lie groups that are very different from each other: 

1. \emph{Nilpotent Lie groups}.  
If $G$ is a nilpotent Lie group with its Lie algebra $\gg$, then the first piece of their corresponding coadjoint dynamical system can be described as follows. 
Let $\{X_1,\dots,X_m\}$ be a Jordan-H\"older basis in $\gg$, 
hence $\gg_k:=\spann\{X_j\mid j=1,\dots,k\}$ is an ideal of $\gg$ for $k=1,\dots,m$. 
Define a total ordering of the subsets of $\{1,\dots,m\}$ by 
$e\preceq e'$ if and only if $\min(e\setminus e')\le\min(e'\setminus e)$, 
with the convention $\min\emptyset:=\infty$, and denote by $e_1$ the smallest set with respect to that ordering. 
If for every $\xi\in\gg^*$ we define 
$J(\xi):=\{j\in\{1,\dots,m\}\mid \gg(\xi)+\gg_{j-1}\subsetneqq\gg(\xi)+\gg_j\}$, 
where $\gg_0:=\{0\}$, then the ``first piece'' from the proof of Theorem~\ref{th_exp} is 
$A_1:=\{\xi\in\gg^*\mid J(\xi)=e_1\}$.

2. \emph{Frobenius Lie groups}.  
A connected Lie group $G$ with its Lie algebra $\gg$ is said to be a Frobenius group if 
there exists $\xi\in\gg^*$ such that the skew-symmetric bilinear form 
$B_\xi\colon\gg\times\gg\to\RR$, $B_\xi(X,Y):=\xi([X,Y])$, is nondegenerate. 
This implies that the center of $\gg$ is equal to $\{0\}$, 
hence the intersection of the classes of Frobenius groups and nilpotent Lie groups is trivial. 
The properties of Frobenius groups and algebras are very interesting and have been studied for a long time; 
see for instance the remarkable paper \cite{Oo74} and references therein. 
Their significance for representation theory comes from the fact that the Frobenius property of $G$ is equivalent to 
the fact that there is at least one open coadjoint orbit.  
If this is the case, there exist finitely many open coadjoint orbits (see \cite[Prop. 4.5]{BB16a} and Section~\ref{Sect6} below). 
If $G$ is moreover an exponential solvable Lie group, then 
the ``first piece'' $A_1$ from the proof of 
Corollary~\ref{th_exp} 
is the union of its open coadjoint orbits. 

In this case, the cross-section $\Xi_1$ is a finite set because it is a complete system of distinct representatives 
of the coadjoint orbits contained in $A_1$. 
The fibers of the group bundle $\Pi_1\colon \Gamma_1\to\Xi_1$ from the proof of Corollary~\ref{th_exp} 
are 0-dimensional, hence it follows 
that the $C^*$-algebra of the groupoid $\Pi_1$ is an $N$-dimensional abelian $C^*$-algebra, 
where $N$ is the number of open coadjoint orbits of $G$. 
This illustrates the significance of the coadjoint dynamical system again,  
because $C^*(G)$ contains an ideal $*$-isomorphic to a direct sum of $N$ copies of the $C^*$-algebra of compact operators on a separable infinite-dimensional complex Hilbert space, as noted in \cite[Rem. 4.7]{BB16a}. 
\end{remark}

We now give another corollary of Theorem~\ref{stratif}, 
which is applicable for unipotent representations. 

\begin{corollary}\label{final3}
Let  $\tau\colon G\to\End(\Vc)$ be a unipotent representation 
of a nilpotent Lie group~$G$. 
Consider the corresponding 
quotient map $q\colon \Vc\to G\setminus\Vc$, $v\mapsto\tau(G)v$, and 
for every $v\in \Vc$ denote $G(v):=\{g\in G\mid\tau(g)v=v\}$. 

Then there exists a partition into $G$-invariant subsets $\Vc=\bigsqcup_{k=1}^n A_k$  
satisfying the following conditions for $k=1,\dots,n$: 
\begin{enumerate}
\item\label{final3_item1} 
The set $D_k:=\bigsqcup_{j=1}^k A_j$ is open in $\Vc$. 
\item\label{final3_item2} 
The map $q\vert_{\Xi_k}\colon\Xi_k\to G\setminus A_k$ is a homeomorphism 
on a suitable locally closed subset $\Xi_k\subseteq A_k$. 
\item\label{final3_item3} 
The dimension of $G$-orbits contained in $A_k$ is constant. 
\item\label{final3_item4} 
If $\Gamma_k:=\bigsqcup\limits_{v\in\Xi_k} \{v\}\times G(v)\subseteq\Xi_k\times G$  
then the map $\Pi_k\colon \Gamma_k\to\Xi_k$, $\Pi_k(v,g)=v$, 
is a group bundle that has local cross-sections for $k=1,\dots,n$, 
and the transformation group $(G,\Vc)$ is a piecewise pullback of group bundles. 
\end{enumerate}
\end{corollary}

\begin{proof}
Existence of the partition $\Vc=\bigsqcup_{k=1}^n A_k$ with the properties \eqref{final3_item1}--\eqref{final3_item3} 
follows by \cite[Th. 3.1.14]{CG90}. 
Then Assertion~\eqref{final3_item4} follows by Theorem~\ref{stratif}. 
\end{proof}

\section{On connected Lie groups with primitive $C^*$-algebras}\label{Sect6}

In this section we discuss two classes of examples of solvable Lie groups that can be studied using the methods developed in the preceding sections. 

The first of these classes consists of Lie groups whose $C^*$-algebras are primitive (that is, they admit faithful irreducible representations) and we briefly discuss the open problem of classifying the Lie groups satisfying that condition 
(see Examples \ref{simplest}, \ref{Pr4} and \ref{Pr5}). 
These examples are obtained from coadjoint dynamical systems of some non-exponential solvable Lie groups 
with a unique open orbit, which is also dense. 
Several interesting results can be found in \cite{Oo74} 
on  the version of the above primitivity problem for universal enveloping algebras of Lie algebras instead of $C^*$-algebras of Lie groups. 

The second class of examples discussed below consists of exponential solvable Lie groups (see Examples \ref{ax+b} and \ref{real}) with several open coadjoint orbits. 

\subsection*{Examples of connected Lie groups whose $C^*$-algebras are primitive}
There has been recent interest in describing discrete groups whose $C^*$-algebras are primitive, 
that is, $C^*$-algebras that have faithful irreducible $*$-representations 
(see \cite{Mu03}, \cite{Om14} and the references therein).  
The similar problem on connected Lie groups seems to be still open,
although it follows by the recent result \cite[Th. A]{Rau16}
that the $C^*$-algebra of every connected Lie group (having its
dimension~$\ge 1$)
admits at least one irreducible $*$-representation which is not faithful.
In connection with that problem, 
and as an application of our groupoid approach to transformation groups, 
we suggest a method to construct examples of connected solvable Lie groups of type~I 
whose $C^*$-algebras are primitive (see Example~\ref{Pr4}), 
and this construction is based on Theorem~\ref{Pr2} below. 

On the negative side, it follows by Proposition~\ref{fred0} that if a Lie group is CCR, 
then its $C^*$-algebra cannot be primitive. 
This is the case for nilpotent or semisimple Lie groups, for instance. 
Using the method of coadjoint orbits for exponential solvable Lie groups (see \cite{LeLu94}), 
it is also easily seen that if $G$ is such a Lie group, 
then its $C^*$-algebra $C^*(G)$ is primitive if and only if $G$ has a unique open coadjoint orbit. 
But this is never the case by Proposition~\ref{even} below. 
In the special case of completely solvable Lie groups, it also follows 
by \cite[Sect. 1]{CuPe89} and \cite[Prop. 4.5]{BB16a}, 
that  there exist either at least two open coadjoint orbits or none. 

\begin{proposition}\label{even}
The number of open coadjoint orbits of any exponential solvable Lie group is finite and even. 
\end{proposition}

\begin{proof}
Let $G$ be an exponential solvable Lie group. 
We first prove that 
\begin{equation}\label{even_proof_eq1}
(\forall \xi\in\gg^*\setminus\{0\})\quad \Oc_{\xi}\ne\Oc_{-\xi}.
\end{equation} 
To this end we reason by contradiction. 
Assuming that for some $\xi\in\gg^*\setminus\{0\}$ one has $\Oc_{\xi}=\Oc_{-\xi}$, that is, $-\xi\in\Oc_{\xi}$, 
it follows that there exists 
$x\in G$ with $\Ad_G^*(x)\xi=-\xi$. 
Then $-1$ is an eigenvalue of the linear operator $\Ad_G^*(x)=(\Ad_G(x)^{-1})^*\colon\gg^*\to\gg^*$, 
hence $-1$ is also an eigenvalue of $\Ad_G(x)\colon\gg\to\gg$. 

On the other hand, let us denote $m:=\dim\gg$ and 
select a sequence of ideals of the complexified Lie algebra $\gg_{\CC}$, 
$$\{0\}=\gg_0\subset\gg_1\subset\cdots\subset\gg_m=\gg_{\CC}$$ 
with $\dim_{\CC}\gg_j=j$ for $j=0,\dots,m$. 
Also select a basis $\{X_1,\dots,X_m\}$ of $\gg_{\CC}$ with $X_j\in\gg_j\setminus\gg_{j-1}$ for $j=1,\dots,m$. 
Then there exist $\CC$-linear functionals $\lambda_1,\dots,\lambda_m\colon\gg_{\CC}\to\CC$ 
for which 
\begin{equation}\label{even_proof_eq2}
(\forall X\in\gg_{\CC})\quad  
\ad_{\gg_{\CC}} X=
\begin{pmatrix}
\lambda_1(X) & \ast & \ast \\
            &\ddots & \ast \\
 \text{\Large 0}           &       & \lambda_m(X)
\end{pmatrix}
\end{equation}
where the above triangular matrix is written with respect to the basis $\{X_1,\dots,X_m\}$ of $\gg_{\CC}$. 
The functionals $\lambda_1,\dots,\lambda_m$ are the roots of the Lie algebra $\gg$, 
hence the hypothesis that the Lie group $G$ is exponential is equivalent to the fact that for 
every $j=1,\dots,m$ there exist the number $\theta_j\in\RR$ and the functional $\xi_j\in\gg^*$ 
with $[\gg,\gg]\subseteq\Ker\xi_j$ and $\lambda_j(X)=(1+\ie\theta_j)\xi_j(X)$ for every $X\in\gg$ 
(see for instance \cite[Subsect. 5.2.2]{FL15}). 
But then it is easily checked that $\ee^{\lambda_j(X)}\ne-1$ for every $j=1,\dots,m$ and $X\in\gg$. 
It then follows by \eqref{even_proof_eq2} that for every $X\in\gg$ 
the operator $\Ad_G(\exp_G X)=\exp(\ad_{\gg} X)\colon\gg\to\gg$ has all its eigenvalues different from~$-1$.
Since the Lie group $G$ is exponential, hence in particular its exponential map $\exp_G\colon\gg\to G$ is surjective, 
we thus obtain that for every $x\in G$ all the eigenvalues of the operator $\Ad_G(x)\colon\gg\to\gg$ 
are different from~$-1$, which is a contradiction with the conclusion of the above paragraph. 

Thus \eqref{even_proof_eq1} is completely proved. 
We now note that the set of open coadjoint orbits of $G$ is finite 
and for every $\xi\in\gg^*$ its coadjoint orbit $\Oc_{\xi}$ is open if and only if $\det(\langle \xi, [Y_j, Y_k] \rangle)_{j, k=1, \dots, m}\ne 0$,   for a  basis $Y_1, \dots, Y_m$ of $\gg$ (see \cite[Prop. 4.5]{BB16a}). 
Hence, denoting by $S$ the set of all open coadjoint orbits of $G$, 
one has the involutive map without fixed points 
$S\to S$, $\Oc_\xi\mapsto\Oc_{-\xi}$.  
This implies that the number of elements of $S$ is even, 
which completes the proof. 
\end{proof}

We now prepare for the construction of some connected solvable Lie groups of type~I 
whose $C^*$-algebras are primitive (see Examples~\ref{simplest} and \ref{Pr4}). 

We use the convention that  
any complex vector space~$\Vc$ is denoted by $\Vc_{\RR}$ when regarded as a real vector space. 
Moreover, $\Vc^*$ denotes the complex vector space of all complex linear functionals $\Vc\to\CC$, 
and $\Vc_{\RR}^*$ stands the real vector space of all real linear functionals $\Vc_{\RR}\to\RR$. 
Similarly, $\End(\Vc)$ is the set of all complex linear maps from $\Vc$ into itself, 
and $\End(\Vc_{\RR})$ is the set of all real linear maps from $\Vc_{\RR}$ into itself, 
hence $\End(\Vc)\subseteq\End(\Vc_{\RR})$. 
Also, for any complex Lie group $G$ (i.e., $G$ is a complex manifold and its group operations are holomorphic maps), 
we denote by $G_{\RR}$ its underlying real Lie group. 

\begin{lemma}\label{Pr1}
Let $\Vc$ be any complex vector space that carries a group representation $\tau\colon G\to\End(\Vc)$, 
which induces the right group actions
$\Vc^*\times G\to\Vc^*$, $(\zeta,g)\mapsto\zeta\circ\tau(g)$, 
and $\Vc_{\RR}^*\times G\to\Vc_{\RR}^*$, $(\xi,g)\mapsto\xi\circ\tau(g)$. 
Then the real-part map 
$$\Re\colon\Vc^*\to\Vc_{\RR}^*,\quad (\Re\zeta)(\cdot):=\frac{1}{2}(\zeta(\cdot)+\overline{\zeta(\cdot)})$$
is a $G$-equivariant $\RR$-linear isomorphism. 
\end{lemma}

\begin{proof}
For any $\xi\in \Vc_{\RR}^*$ define $\widetilde{\xi}\in\Vc_{\RR}^*$ by $\widetilde{\xi}(v):=-\xi(\ie v)$ for every $v\in\Vc$. 
With this notation it is easily checked that for $\xi,\eta\in \Vc_{\RR}^*$ we have $\xi+\ie\eta\in \Vc^*$ 
if and only if $\eta=\widetilde{\xi}$. 
This implies that the map $\Re\colon\Vc^*\to\Vc_{\RR}^*$ is bijective with its inverse 
$\Re^{-1}\colon \Vc_{\RR}^*\to\Vc^*$, $\xi\mapsto\xi+\ie\widetilde{\xi}$. 
It is straightforward to check that $\Re$ is $\RR$-linear. 
To see that $\Re$ is also $G$-equivariant, it suffices to check that for arbitrary $g\in G$ and $\xi\in\Vc_{\RR}^*$  one has 
$\Re^{-1}(\xi\circ\tau(g))=(\Re^{-1}\xi)\circ\tau(g))$, which is equivalent to $\widetilde{\xi\circ\tau(g)}=\widetilde{\xi}\circ\tau(g)$. 
This last equality follows by the fact that 
$\tau(g)\in\End(\Vc)$, that is, $\tau(g)\in\End(\Vc_{\RR})$ and $\tau(g)(\ie v)=\ie\tau(g)v$ for all $v\in\Vc$. 
\end{proof}

In the setting of Lemma~\ref{Pr1} we denote by $G\ltimes \Vc_{\RR}^*$ the semidirect product of groups, with its multiplication map 
\begin{equation}\label{Pr_eq1}
(g_1,\xi_1)(g_2,\xi_2):=(g_1g_2,\xi_1\circ\tau(g_2)+\xi_2)
\end{equation}
for all $\xi_1,\xi_2\in\Vc_{\RR}^*$ and $g_1,g_2\in G$. 

\begin{theorem}\label{Pr2}
Let $\Vc$ be a finite-dimensional complex vector space that carries a continuous representation $\tau\colon G\to\End(\Vc)$ 
of a locally compact group~$G$. 
Let $v_0\in\Vc$ and $G(v_0):=\{g\in G\mid\tau(g)v_0=v_0\}$. 
Assume that the homogeneous space $G/G(v_0)$ has a $G$-invariant Radon measure and denote by $\mu$ its pushforward  
through the $G$-equivariant continuous map $G/G(v_0)\to \Vc$, $gG(v_0)\mapsto\tau(g)$. 

Then the following assertions hold: 
\begin{enumerate}[(i)]
\item\label{Pr2_item1} The map 
$$\pi\colon G\ltimes\Vc_{\RR}^* \to\Bc(L^2(\Vc,\mu)),\quad 
(\pi(g,\xi)\varphi)(v):=\ee^{\ie\xi(v)}\varphi(\tau(g^{-1})v)$$
is a continuous unitary irreducible representation. 
\item\label{Pr2_item2} If the locally compact group $G$ is amenable, the $G$-orbit $\Oc:=\tau(G)v_0$ is dense in $\Vc$, 
and $G(v_0)=\{\1\}$, 
then  
$\pi\colon C^*(G\ltimes \Vc_{\RR}^*)\to\Bc(L^2(\Vc,\mu))$ is a faithful irreducible $*$-representation. 
\end{enumerate}
\end{theorem}

\begin{proof}
\eqref{Pr2_item1} Since $G$ acts transitively on $\Oc$ and
$\mu(\Vc\setminus\Oc)=0$,
it is easily checked that
the measure $\mu$ is $G$-invariant and ergodic.
Then the group representation $\pi$ is irreducible by \cite[Prop. 3.1]{BB16b}.

\eqref{Pr2_item2}
The fact that the above unitary representation $\pi$ is irreducible
implies that also its corresponding $*$-representation
$\pi\colon C^*(G\ltimes \Vc_{\RR}^*)\to\Bc(L^2(\Vc,\mu))$ is irreducible.

Let $f_1\in \Cc_c(G)$, $f_2\in\Cc_c(\Vc_{\RR}^*)$, and define $f:=f_1\otimes f_2\in L^1(G\ltimes \Vc_{\RR}^*)$. 
It is well known and straightforward to check using \eqref{Pr_eq1} 
that if we denote by $\de g$ any left Haar measure of $G$ and by $\de\xi$ any Lebesgue measure of $\Vc_{\RR}^*$, then the direct product of these measures is a left Haar measure of the semidirect product $G\ltimes \Vc_{\RR}^*$. 
Then we have for all $\varphi\in \Cc_c(\Vc)$ and $v\in\Vc$, 
$$\begin{aligned}
(\pi(f_1\otimes f_2)\varphi)(v)
:=
&\int\limits_G\int\limits_{\Vc_{\RR}^*} f_1(g)f_2(\xi)\ee^{\ie\xi(v)}\varphi(\tau(g^{-1})v)\de g\de\xi \\
=
& (\Fc^{-1} f_2)(v)\int\limits_G f_1(g)\varphi(\tau(g^{-1})v)\de g
\end{aligned}$$
hence 
\begin{equation}\label{Pr2_proof_eq1}
\pi(f_1\otimes f_2)\varphi
=\Fc^{-1} f_2\cdot\int\limits_G f_1(g)\cdot \varphi\circ\tau(g^{-1})\de g
\end{equation}
where $\Fc\colon \Sc(\Vc_{\RR})\to\Sc(\Vc_{\RR}^*)$ is the Fourier transform. 

On the other hand, one has the $C^*$-dynamical system $(\Cc_0(\Vc_{\RR}),G,\alpha)$, 
where
$$\alpha\colon G\times\Cc_0(\Vc_{\RR})\to\Cc_0(\Vc_{\RR}),\quad \alpha_g(h):=h\circ\tau(g^{-1}).$$ 
Since the measure $\mu$ on $\Vc_{\RR}$ is invariant under the group action of $G\times\Vc_{\RR}\to\Vc_{\RR}$, 
$(g,v)\mapsto\tau(g)v$, we obtain a covariant representation $(M,\sigma)$ of  $(\Cc_0(\Vc_{\RR}),G,\alpha)$ 
on $L^2(\Vc_{\RR},\mu)$, where we use the $*$-representation 
$$M\colon \Cc_0(\Vc_{\RR})\to\Bc(L^2(\Vc_{\RR},\mu)), \quad M(h)\varphi:=h\varphi,$$ 
and the unitary representation 
$$\sigma\colon G\to\Bc(L^2(\Vc_{\RR},\mu)), \quad \sigma(g)\varphi:=\varphi\circ\tau(g^{-1}).$$  
It is clear that $M(\alpha_g(h))=\sigma(g)M(h)\sigma(g)^{-1}$ for all $g\in G$ and $h\in\Cc_0(\Vc_{\RR})$, 
hence the pair $(M,\sigma)$ is indeed a covariant representation of the $C^*$-dynamical system $(\Cc_0(\Vc_{\RR}),G,\alpha)$, 
and then that pair defines a $*$-representation of the crossed product, 
$\sigma\ltimes M\colon G\ltimes \Cc_0(\Vc_{\RR})\to \Bc(L^2(\Vc_{\RR},\mu))$, 
given by 
\begin{equation}\label{Pr2_proof_eq2}
(\sigma\ltimes M)(f_1\otimes f_2)=M(f_2)\int\limits_G f_1(g)\sigma(g)\de g
\end{equation}
for arbitrary $f_1\in L^1(G)$ and $f_2\in \Cc_c(\Vc_{\RR})$. 

Comparing \eqref{Pr2_proof_eq1} and \eqref{Pr2_proof_eq2}, we obtain 
\begin{equation}\label{Pr2_proof_eq3}
\pi\circ\Psi=\sigma\ltimes M
\end{equation}
for the $*$-morphism 
$$\Psi\colon G\ltimes \Cc_0(\Vc_{\RR})\to C^*(G\ltimes \Vc_{\RR}^*),\quad 
\Psi(f_1\otimes f_2)= f_1\otimes \Fc f_2$$
for all $f_1\in L^1(G)$ and $f_2\in\Cc_c(\Vc_{\RR})$. 
But it is well known that, considering the $C^*$-algebra $C^*(\Vc_{\RR}^*)$ of the abelian group $(\Vc_{\RR}^*,+)$,  
the Fourier transform $\Fc$ extends to 
a $*$-isomorphism $\Cc_0(\Vc_{\RR})\to C^*(\Vc_{\RR}^*)$ 
and moreover one has the canonical isomorphism $C^*(G\ltimes \Vc_{\RR}^*)\simeq G\ltimes C^*(\Vc_{\RR}^*)$ 
(see for instance \cite[Prop. 3.11]{Wi07}). 
Therefore the above map $\Psi$ is a $*$-isomorphism.  

Since the locally compact group $G$ is amenable, it follows by \cite[Prop. 1.13]{Pa88} that 
also the group $G\ltimes \Vc_{\RR}^*$ is amenable, hence one has the canonical $*$-isomorphism 
$C^*(G\ltimes \Vc_{\RR}^*)\simeq C^*_{\rm red}(G\ltimes \Vc_{\RR}^*)$ (see \cite[Th. 4.21]{Pa88}). 
The hypothesis that $G$ is amenable also implies that 
if $\Gc:=G\times\Vc\tto\Vc$ is the groupoid associated with the transformation group $(G,\Vc)$, 
with its domain/range maps $d,r\colon G\times\Vc\to\Vc$, $d(g,v):=v$, $r(g,v):=\tau(g)v$, 
then $\Gc$ is metrically amenable, that is,  we have $C^*(\Gc)\simeq C^*_{\rm red}(\Gc)$ as in the definition before Remark~\ref{rem2.1}.
Hence, also using the hypothesis that the orbit $\Oc$ is dense, it follows by Proposition~\ref{dense2} 
that the regular representation $\pi_{v_0}\colon C^*(\Gc)\to\Bc(L^2(\Gc_{v_0}))$ is faithful. 
Here 
$$\Gc_{v_0}=\{(g,v)\in G\times\Vc\mid d(g,v)=v_0\}=\{(g,v_0)\in G\times\Vc\mid g\in G\}$$
Using the hypothesis $G(v_0)=\{\1\}$,  
it is straightforward to check that the map 
$$r\vert_{\Gc_{v_0}}\colon \Gc_{v_0}\to\Oc\hookrightarrow\Vc_{\RR}, \quad (g,v_0)\mapsto\tau(g)v_0,$$ 
defines a unitary operator 
$L^2(\Gc_{v_0})\to L^2(\Vc_{\RR},\mu)$ that intertwines the representations 
$\pi_{v_0}$ and $\sigma\ltimes M$ (compare for instance \cite[Th. 1.5]{Gl62}), 
so the $*$-representation $\sigma\ltimes M$ is also faithful.  

Now by \eqref{Pr2_proof_eq3} 
along with the fact that $\Psi$ is a $*$-isomorphism we obtain that the representation $\pi$ is faithful. 
This completes the proof. 
\end{proof}

We now discuss the simplest setting in which Theorem~\ref{Pr2} applies. 

\begin{example}\label{simplest}
\normalfont
For the multiplicative group $G:=(\CC\setminus\{0\},\cdot )$ and the complex 1-dimensional vector space $\Vc:=\CC$  
we define the representation 
$$\tau\colon G\to\End(\Vc),\ \tau(a)v:=av.$$
Then the hypothesis of Theorem~\ref{Pr2} is satisfied for $v_0=1\in\CC$, since $G(v_0)=\{1\}$ 
and the orbit $\tau(G)v_0=\CC\setminus\{0\}$ is a dense open subset of $\CC$. 
It follows that the $C^*$-algebra of the complex $ax+b$-group $G\ltimes\Vc_{\RR}^*=(\CC\setminus\{0\})\ltimes\CC$ 
has a faithful irreducible representation, which is actually well known (see \cite[\S 2]{Ro76} and \cite{ET79}). 
This group is a connected solvable Lie group, which is not simply connected. 
\end{example}

For  Example~\ref{Pr4})
we need 
a few simple facts that we establish in Lemmas \ref{alg1}--\ref{alg3}. 

Let $\Vc$ be a finite-dimensional real vector space with its group of invertible endomorphisms denoted by~$\End(\Vc)^\times$. 
A subgroup $G\subseteq \End(\Vc)^\times$ is called an \emph{algebraic group} if there exists a family $\Pc$ 
of real-valued polynomial functions on the real vector space $\End(\Vc)\times \End(\Vc)$ for which 
$G=\{g\in\End(\Vc)^\times\mid(\forall p\in\Pc)\ p(g,g^{-1})=0\}$. 
Such an algebraic group $G$ is clearly a closed subgroup of $\End(\Vc)^\times$, 
so $G$ is a Lie group with its topology induced from~$\End(\Vc)^\times$. 

A Lie subalgebra $\gg\subseteq\End(\Vc)$ is called an \emph{algebraic Lie algebra} if 
the subgroup of $\End(\Vc)^\times$ generated by $\{\exp X\mid X\in\gg\}$ is 
the connected component of $\1$ in an algebraic group.  
See \cite[subsection 2.1 a)]{Pu69} for a discussion on why this definition is equivalent to the definition in terms of replicas.

\begin{lemma}\label{alg1}
Let $\Vc$ be a finite-dimensional real vector space with a solvable Lie algebra $\gg\subseteq\End(\Vc)$. 
Denote $k:=\dim_{\RR}\Vc$ and $\n:=\{X\in\gg\mid X^k=0\}$. 
Then $\gg$ is an algebraic Lie algebra if and only if there exists an abelian subalgebra $\ag\subseteq\gg$ satisfying the following conditions: 
\begin{enumerate}
\item One has the direct sum decomposition $\gg=\ag\dotplus\n$. 
\item There exists a basis $\{X_1,\dots,X_r\}$ of $\ag_{\CC}:=\ag\otimes_{\RR}\CC\hookrightarrow\End(\Vc_{\CC})$ 
such that for $j=1,\dots,k$ the operator $X_j\in\End(\Vc_{\CC})$ is diagonalizable and its spectrum consists of integer numbers. 
\end{enumerate}
\end{lemma}

\begin{proof}
See \cite[subsection 2.1 a)]{Pu69}. 
\end{proof}

\begin{lemma}\label{alg2}
If $G$ is a connected simply connected solvable Lie group for which the Lie algebra $\ad_{\gg}(\gg)$ is algebraic, 
then both $G$ and its tangent group $TG$ are of type~I. 
In particular, if $B$ is the connected simply connected Lie group corresponding to a Borel subalgebra of a complex semisimple Lie algebra, then both $B_{\RR}$ and its tangent group $T(B_{\RR})$  
are solvable Lie groups of type~I. 
\end{lemma}

\begin{proof}
It follows by \cite[Prop. 2.2]{Pu69} that $G$ is a Lie group of type~I. 
The tangent group $\widetilde{G}:=TG=G\ltimes\gg$ has its Lie algebra 
$\widetilde{\gg}:=\gg\ltimes\gg$, whose Lie bracket is given by 
$[(X_1,Y_1),(X_2,Y_2)]=([X_1,X_2],[X_1,Y_2]-[X_2,Y_1])$ for all $X_1,X_2,Y_1,Y_2\in\gg$. 
Since $\ad_{\gg}(\gg)$ is an algebraic Lie algebra of endomorphisms of the linear space $\gg$, 
one has a direct sum decomposition $\ad_{\gg}(\gg)=\ag\dotplus\n$ as in Lemma~\ref{alg1}. 
Then the above formula for the Lie bracket of $\widetilde{\gg}$ shows that 
$\ad_{\widetilde{\gg}}(\widetilde{\gg})=\widetilde{\ag}\dotplus\widetilde{\gg}$, 
where $\widetilde{\ag}:=\{\ad_{\widetilde{\gg}}(X,0)\mid X\in\ag\}$ is an abelian Lie algebra 
of diagonalizable endomorphisms of the linear space $\gg$ having integer spectra, and  
$\widetilde{\n}:=\{\ad_{\widetilde{\gg}}(X,Y)\mid X\in\n,Y\in\gg\}$ consists of 
nilpotent endomorphisms of the linear space $\gg$. 
Hence $\widetilde{G}$ is a connected simply connected solvable Lie group for which 
the Lie algebra $\ad_{\widetilde{\gg}}(\widetilde{\gg})$ is algebraic by Lemma~\ref{alg1}, 
and then $\widetilde{G}$ is a Lie group of type~I by by \cite[Prop. 2.2]{Pu69} again. 

For the second assertion, we use the notation of Example~\ref{Pr4} 
and we check that the Lie algebra $\ad_{\gg}(\gg)$ is algebraic. 
We have $\gg=\hg\dotplus\n^{+}$, hence $\ad_{\gg}(\gg)=\ad_{\gg}(\hg)+\ad_{\gg}(\n^{+})$. 
Here  $\ad_{\gg}(\hg)$ consists of diagonalizable linear maps on~$\gg$ while $\ad_{\gg}(\n^{+})$ consists of nilpotent maps, 
hence $\ad_{\gg}(\hg)\cap\ad_{\gg}(\n^{+})=\{0\}$. 
Moreover, it follows by the structure theory of complex semisimple Lie algebras 
that $\ad_{\gg}(\hg)$ has a basis consisting of maps whose spectra contain only integer numbers. 
It then follows by Lemma~\ref{alg1} that $\ad_{\gg}(\gg)$  is an algebraic Lie algebra, and we are done. 
\end{proof}

\begin{lemma}\label{alg3}
For any connected complex Lie group $G$, its corresponding real Lie group $G_{\RR}$ is an exponential Lie group 
if and only if $G$ is simply connected and nilpotent. 
\end{lemma}

\begin{proof}
See \cite[Cor.~of Th.~1, p.~118]{Dix57}.
  \end{proof}

We illustrate Theorem~\ref{Pr2} by an example.
The role of the group $G$ from Theorem~\ref{Pr2} is played by 
$\widetilde{G}/Z(\widetilde{G})$  for Borel subgroups $\widetilde{G}$  of some complex simple Lie groups, where $Z(\widetilde{G})$ is the center of $\widetilde{G}$. 
For this reason we use somewhat unusual notation in structure theory of semisimple Lie groups (see \cite{Kn02}) 
in the sense that we keep the notation $G$ for a quotient group $\widetilde{G}/Z(\widetilde{G})$ as above. 

\begin{example}\label{Pr4}
\normalfont
Let $\sg$ be any complex semisimple Lie algebra, hence its Killing form 
$$B_{\sg}\colon\sg\times\sg\to\mathbb C,\quad B_{\sg}(X,Y):=\Tr((\ad_{\sg}X)(\ad_{\sg}Y))$$
is nondegenerate, 
where for every $X\in\sg$ we define as usual $\ad_{\sg}X\colon\sg\to\sg$, $(\ad_{\sg}X)Y=[X,Y]$ for $X,Y\in\sg$.  
Then there exists a conjugate-linear mapping $\sg\to\sg$, $X\mapsto X^*$, 
such that for all $X,Y\in\sg$ we have $(X^*)^*=X$, $[X,Y]^*=[Y^*,X^*]$ and $B_{\sg}(X,X^*)\ge 0$. 
Such a mapping 
is unique up to an automorphism of~$\sg$. 

Select any Cartan subalgebra $\hg\subseteq\sg$, that is, $\hg$ is a maximal abelian self-adjoint subalgebra of $\sg$. 
Denote the eigenspaces for the adjoint action of $\hg$ on $\sg$ by 
$$\sg^\alpha:=\{X\in\sg\mid(\forall H\in\hg)\quad [H,X]=\alpha(H)X\}
\text{ for a linear functional }\alpha\colon\hg\to\mathbb C$$
and consider the corresponding set of roots
$\Delta(\sg,\hg):=\{\alpha\in\hg^*\setminus\{0\}\mid \sg^\alpha\ne\{0\}\}$. 
Then $\hg=\sg^0=\{X\in\sg\mid[\hg,X]=\{0\}\}$ and the root space decomposition of $\hg$ is given by 
$$\sg=\hg\oplus\bigoplus_{\alpha\in\Delta(\sg,\hg)}\sg^\alpha.$$ 
For every $\alpha\in\Delta(\sg,\hg)$ we have $\dim\sg^\alpha=1$ 
and we may choose $X_\alpha\in\sg^\alpha\setminus\{0\}$ such that $X_\alpha^*=X_{-\alpha}$.

Now fix a basis $\{H_1,\dots,H_m\}$ of $\hg$ with $H_j=H_j^*$ for $j=1,\dots,m$, and define 
$\Delta^+(\sg,\hg)$ as the set of all linear functionals $\varphi\colon\hg\to\CC$ with $\varphi(H^*)=\overline{\varphi(H)}$ for all $H\in\hg$ and for wich there exists $k\in\{1,\dots,m\}$ with $\varphi(H_j)=0$ if $1\le j\le k-1$ and $\varphi(H_k)>0$. 
One can prove that $\Delta(\sg,\hg)=\Delta^{+}(\sg,\hg)\sqcup(-\Delta^+(\sg,\hg))$. 
The elements of the set  $\alpha\in\Delta^{+}(\sg,\hg)\setminus(\Delta^{+}(\sg,\hg)+\Delta^{+}(\sg,\hg))$ are called simple roots. 
We label the simple roots as $\{\alpha_1,\dots,\alpha_r\}$, 
and one can prove that this set is a basis of the dual space $\hg^*$, hence $r=\dim\hg$, and this is called the rank of $\sg$ 
(see for instance \cite{Kn02}).   
The above root space decomposition of $\sg$ implies the triangular decomposition 
$$\sg=\n^{-}\dotplus\hg\dotplus\n^{+}$$
where $\n^{\pm}:=\bigoplus\limits_{\pm\alpha\in\Delta^+(\sg,\hg)}\sg^\alpha$. 

We are interested in the Lie algebra 
$$\gg:=\hg\dotplus\n^{+}=\hg\ltimes\n^{+},$$ 
which is a Borel subalgebra of $\sg$. 
This is a solvable complex Lie algebra,
and we denote by $\widetilde{G}$ its corresponding connected subgroup
of
the connected, simply connected Lie group associated to~$\sg$.
It is well known and easily checked that the center of $\gg$ is trivial and 
the center $Z(\widetilde{G})$ of $\widetilde{G}$ is finite.
We define $G:=\widetilde{G}/Z(\widetilde{G})$, which is then a Lie group
whose Lie algebra is again $\gg$.  
We note however that $G$  is not an exponential Lie group (by Lemma~\ref{alg3}). 

It follows by \cite[Th. 1.7]{Ko12} that the group $\widetilde{G}$ has a nonempty open coadjoint orbit if and only if the rank of $\sg$ is equal to the cardinality of a maximal set $\Bc$ of strongly orthogonal roots obtained by the cascade construction. 
If this is the case, then the open orbit is unique hence it is dense in $\sg^*$.  
More specifically, for every $\beta\in\Bc$ we select a non-zero vector $e_{-\beta}\in\n^{-}$ in the root space corresponding to $-\beta$, and we define 
$$\rr_{-}^\times:=
\Bigl\{\sum_{\beta\in\Bc} t_\beta e_{-\beta}\mid (\forall \beta\in\Bc)\quad t_\beta\in\CC^\times\Bigr\}. $$
Using the linear isomorphism $\gg^*\simeq\n^{-}\dotplus\hg$ 
defined via the Killing form of $\sg$, we consider the coadjoint action of $\widetilde{G}$, 
$$\Ad_{\widetilde{G}}^*\colon \widetilde{G}\times(\n^{-}\dotplus\hg)\to \n^{-}\dotplus\hg$$
and we define the set 
$$\Oc:=\bigcup_{\xi\in\rr_{-}^\times} \Ad_{\widetilde{G}}^*(N)\xi+\hg\subseteq \n^{-}\dotplus\hg.$$
If we assume $\vert \Bc\vert=\dim\hg$, then the following assertions hold: 
\begin{enumerate}[(i)]
	\item\label{center2_item1} 
	The set $\Oc$ is an open dense coadjoint orbit of the Borel subgroup~$\widetilde{G}$ of $S$. 
	\item\label{center2_item2}  For every $\xi\in\rr_{-}^\times$ the corresponding coadjoint isotropy group $\widetilde{G}(\xi)$, which is a discrete subgroup of $\widetilde{G}$, is equal to 
	$$H^{\Bc}:=\{h\in H\mid (\forall \beta\in\Bc)\quad \Ad_{\widetilde{G}}^*(h)e_{-\beta}=e_{-\beta}\}$$
\end{enumerate}
Indeed, Assertion~\eqref{center2_item1} follows by \cite[Th. 2.7]{Ko12}. 
The proof of Assertion~\eqref{center2_item2} requires several steps:

Step 1: For all $\xi\in\rr_{-}^\times$ one has $\widetilde{G}(\xi)\cap H=H^{\Bc}$. 
In fact, if $\xi=\sum\limits_{\beta\in\Bc} t_\beta e_{-\beta}$, 
and $h\in H$, then, using the fact that $e_{-\beta}$ is an eigenvector of $\Ad_{\widetilde{G}}^*(h)$ and $t_\beta\ne0$ for all $\beta\in\Bc$, 
we obtain $\Ad_{\widetilde{G}}^*(h)\xi=\xi$ if and only if $h\in H^{\Bc}$. 

Step 2: We now prove that if $\xi\in \rr_{-}^\times$, 
$h\in H$, $n\in N^{+}$, and $nh\in \widetilde{G}(\xi)$, then $h\in H^{\Bc}\subseteq \widetilde{G}(\xi)$ and $n\in \widetilde{G}(\xi)$. 
Indeed, one has $\Ad_{\widetilde{G}}^*(n)(\Ad_{\widetilde{G}}^*(h)\xi)
=\Ad_{\widetilde{G}}^*(nh)\xi=\xi=\Ad_{\widetilde{G}}^*(\1)\xi$ 
with $\Ad_{\widetilde{G}}^*(h)\xi,\xi\in\rr_{-}^\times$, 
hence by \cite[Th. 1.1]{Ko12} we obtain $\Ad_{\widetilde{G}}^*(h)\xi=\xi$. 
Now $h\in H^{\Bc}\subseteq \widetilde{G}(\xi)$ by Step 1. 

Step 3: For all $\xi\in\rr_{-}^\times$ one has  $\widetilde{G}(\xi)=H^{\Bc}$. 
In fact, 
since $\Ad_{\widetilde{G}}^*\vert_{N^{+}}$ is a unipotent representation, 
it follows that $N^{+}\cap \widetilde{G}(\xi)$ is connected 
(by \cite[Lemma 3.1.1]{CG90}), 
hence $N^{+}\cap \widetilde{G}(\xi)=\{\1\}$. 
Then, by Step 2, we obtain $\widetilde{G}(\xi)\subseteq H^{\Bc}$, 
and further by Step 1, $\widetilde{G}(\xi)= H^{\Bc}$. 

Let $\widetilde{G}$ be a Borel subgroup of the connected simply connected Lie group $S$ corresponding to a complex semisimple Lie algebra $\sg$ as above. 
\textit{
In addition to the condition $\vert\Bc\vert=\dim\hg$, 
let us also assume the following: 
\begin{equation}\label{ZG}
H^{\Bc}=Z(\widetilde{G}), 
\end{equation}}
which is clearly equivalent to $H^{\Bc}=\{h\in H\mid (\forall \beta\in\Pi)\quad \Ad_{\widetilde{G}}^*(h)e_{-\beta}=e_{-\beta}\}$, 
for the system of simple roots $\Pi\subseteq\Delta^{+}(\sg,\hg)$. 

It is well known and easily checked that the center of $\gg$ is trivial and 
the center $Z(\widetilde{G})$ of $\widetilde{G}$ is finite.
We define $G:=\widetilde{G}/Z(\widetilde{G})$, which is then a Lie group
whose Lie algebra is again $\gg$.  
However $G$  is not an exponential Lie group (by Lemma~\ref{alg3}).

For any group $G$ as above, let us denote by $\Vc:=\gg^*$ 
the complex vector space of all $\CC$-linear functionals $v\colon\gg\to\CC$, using the notation convention prior to Lemma~\ref{Pr1}. 
If we define $\tau\colon G\to\Bc(\Vc)$, $\tau(g)v:=v\circ\Ad_G(g^{-1})$ (the complex coadjoint representation of $G$), 
then for every $v_0\in\rr_{-}^\times\subseteq \Vc$ its orbit $\Oc:=\tau(G)v_0$ is a dense open subset of $\Vc$. 
We obtain $G(v_0)=\{\1\}$ since $\widetilde{G}(v_0)=H^\Bc=Z(\widetilde{G})$ 
by Assertion~\eqref{center2_item2} and assumption~\eqref{ZG}. 
Hence the hypothesis of Theorem~\ref{Pr2}  is satisfied and it follows that $G\ltimes\Vc_{\RR}^*$ is a Lie group whose $C^*$-algebra has a faithful irreducible $*$-representation. 
It is easily seen that if we denote by $G_{\RR}$ the complex Lie group $G$ regarded as a real Lie group, 
with the Lie algebra of $G_{\RR}$ denoted by $\gg_{\RR}$, then $G\ltimes\Vc_{\RR}^*=G_{\RR}\ltimes\gg_{\RR}=T(G_{\RR})$ 
is the tangent group of $G_{\RR}$. 
Thus \emph{$T(G_{\RR})$ is a connected solvable Lie group (hence amenable) whose $C^*$-algebra is primitive}. 
Moreover, using Lemma~\ref{alg2}, 
it follows that $T(G_{\RR})$ is a group of type~I. 
However, $T(G_{\RR})$ is not an exponential Lie group, as a consequence of Lemma~\ref{alg3}. 

By \cite[Rem. 2.8]{Ko12} or \cite[Prop. 4.2(1)]{AP97}, 
if $\sg$ is a simple Lie algebra, then it satisfies the above condition $\vert\Bc\vert=\dim\hg$ if and only if 
it is a classical Lie algebra of type $\so(2\ell+1,\CC)$, $\ssp(2\ell,\CC)$, $\so(2\ell,\CC)$ with $\ell$ even, 
or an exceptional Lie algebra of type $E_7$, $E_8$, $F_4$ or $G_2$. 
We will give below an example  
where the other condition \eqref{ZG}~is trivially satisfied. 
It would be interesting to see which other simple algebras from the above list satisfy condition \eqref{ZG} as well.
\end{example}

\begin{example}\label{Pr5}
\normalfont
Let
 $\sg=\ssp(2,\CC)$, 
which is isomorphic to the Lie algebra 
$$\ssl(2,\CC)=\Bigl\{\begin{pmatrix} a & b \\
c & -a
\end{pmatrix}\mid a,b,c\in\CC\Bigr\}.$$ 
Then one has with the above notation 
$$\hg=\Bigl\{\begin{pmatrix} a & 0 \\
0 & -a
\end{pmatrix}\mid a\in\CC\Bigr\}, \qquad
\n^{+}=\Bigl\{\begin{pmatrix} 0 & b \\
0 & 0
\end{pmatrix}\mid b\in\CC\Bigr\}, $$ 
and then 
$$\gg=\Bigl\{\begin{pmatrix} a & b \\
0 & -a
\end{pmatrix}\mid a,b\in\CC\Bigr\}.$$
The connected subgroup of ${\rm SL}(2,\CC)$ whose Lie algebra is $\gg$ is
$$\widetilde{G}=\Bigl\{\begin{pmatrix} a & b \\
0 & a^{-1}
\end{pmatrix}\mid a,b\in\CC,\ a\ne0\Bigr\}$$
hence a connected complex Lie group of complex dimension 2 whose
center is $Z(\widetilde{G})=\{\pm\1\}$.
We also have $\Pi=\Bc=\Delta^{+}(\gg,\hg)$ 
(since these sets contain only one element)
hence \eqref{ZG} is automatically satisfied. 

One can also describe $\widetilde{G}$ as the semidirect product
$$\widetilde{G}=(\CC\setminus\{0\})\ltimes\CC,\quad
(a_1,b_1)\cdot(a_2,b_2):=(a_1a_2,a_1b_2+b_1a_2^{-1})$$
for all $(a_1,b_1),(a_2,b_2)\in \widetilde{G}$.
This group $\widetilde{G}$ is isomorphic to the $ax+b$-group over the
field~$\CC$ from Example~\ref{simplest}.
Now define $G:=\widetilde{G}/Z(\widetilde{G})$.
It follows by Example~\ref{Pr4} that for the above Lie group~$G$ the $C^*$-algebra of its tangent group $T(G_{\RR})$ is primitive, 
because $G$ has a unique open dense coadjoint orbit, 
unlike the $ax+b$-group over the field~$\RR$, which has two open coadjoint orbits as one can see in Example~\ref{ax+b}. 
We also note that $T(G_{\RR})$ is a connected 8-dimensional solvable Lie group.  
\end{example}

\subsection*{Examples of exponential solvable Lie groups}

We now give one of the simplest examples of Lie groups that fall under the hypotheses of Theorem~\ref{th_exp}. 
Many other examples of trans\-formation-group groupoids can be constructed following this pattern. 
See for instance Example~\ref{real} below. 

\begin{example}[the $ax+b$-group]\label{ax+b}
\normalfont
Let 
$$G=\left\{\begin{pmatrix} a & b \\
                           0 & a^{-1}  
           \end{pmatrix}\in M_2(\RR)\mid  a>0\right\}$$
which is a multiplicative group of matrices. 
For the tautological action of $G$ on~$\RR^2$ denote by $G_{x,y}$ and $\Oc_{x,y}$ the isotropy group and the $G$-orbit 
of any $\begin{pmatrix} x\\ y \end{pmatrix}\in\RR^2$. 
Then one has 
\begin{itemize}
\item if $y\ne0$, then $G_{x,y}=\{\1\}$ and $\Oc_{x,y}=\left\{\begin{pmatrix} u\\ v \end{pmatrix}\in\RR^2 \mid uy>0\right\}$; 
\item if $y=0\ne x$, then 
$$G_{x,0}=\Bigl\{\begin{pmatrix} 1 & b \\
                           0 & 1  
           \end{pmatrix}\mid  b\in\RR\Bigr\} 
\text{ and }\Oc_{x,0}=\left\{\begin{pmatrix} u\\ 0 \end{pmatrix}\in\RR^2 \mid ux>0\right\};$$  
\item if $y=x=0$, then $G_{0,0}=G$ 
 and $\Oc_{0,0}=\left\{\begin{pmatrix} 0\\ 0 \end{pmatrix}\right\}$. 
\end{itemize}
Consequently, if we define 
$$M_{\pm}:=\left\{\begin{pmatrix} x\\ y \end{pmatrix}\in\RR^2 \mid \pm y\ge 0, \ x^2+y^2\ne0\right\}
\supset \left\{\begin{pmatrix} x\\ y \end{pmatrix}\in\RR^2 \mid \pm y> 0\right\}=:U_{\pm}$$
then $M_{\pm}$ is acted on by the group $G$ with dense open orbit $U_{\pm}$ 
and the isotropy groups at points of the boundary $F:=M_{\pm}\setminus U_{\pm}$ 
are all isomorphic to $(\RR,+)$. 

Hence the transformation-group groupoids $\Gc_{\pm}:=G\ltimes M_{\pm}\tto M_{\pm}$ 
have dense open orbits $U_{\pm}$ with $(\Gc_{\pm})_{U_{\pm}}=U_{\pm}\times U_{\pm}$, 
and moreover $(\Gc_{\pm})_{F}=\RR\times F\to F$ is a trivial group bundle on $F=\RR\setminus\{0\}$. 
We get thus immediately the following short exact sequence
$$ 0 \rightarrow \Kc(L^2(U_{\pm})) \rightarrow C^*(\Gc_{\pm}) \rightarrow \Cc_0(\RR\setminus\{0\}, \Cc_0(\RR))\rightarrow 0.$$
\end{example}

In the following example we describe a more general situation, which recovers Example~\ref{ax+b} for $\sg=\su(1,1)$. 

\begin{example}\label{real}
\normalfont
Now let $\sg=\kg+\pg$ be a Cartan decomposition of a real semisimple Lie algebra. 
Select any maximal abelian subspace $\ag\subseteq\pg$, consider the corresponding restricted-root spaces and define  
$$\n=\bigoplus_{\alpha\in\Delta^{+}(\sg,\ag)}\sg^\alpha$$
with respect to a lexicographic ordering on $\ag^*$. 
Then one has the Iwasawa decomposition $\sg=\kg\dotplus\ag\dotplus\n$. 

We are mainly interested in the Lie algebra 
$$\gg:=\ag\dotplus\n=\ag\ltimes\n,$$ 
which is a Borel subalgebra of $\sg$. 
This is a completely solvable Lie algebra, hence its corresponding connected, simply connected Lie group~$G$ is an exponential Lie group. 
Moreover, denoting $r:=\dim\ag$, one can prove that if $\sg$ is of hermitian type, 
then $G$ has $2^r$ open coadjoint orbits 
(see \cite[Prop. 3.3.1]{RV76} and \cite[Th. 1.3]{In01}). 
\end{example}

\subsection*{Acknowledgment}

We thank Karl-Hermann Neeb for some useful remarks on the $ax+b$-groups. 
We also thank the Referees for numerous most useful comments, which have resulted in a considerable improvement of our manuscript.

\end{document}